\documentclass{article}
\usepackage{amsfonts}
\usepackage{amsmath}
\usepackage{shuffle}
\usepackage{scalerel,stackengine}

\stackMath
\newcommand\reallywidehat[1]{\savestack{\tmpbox}{\stretchto{  \scaleto{    \scalerel*[\widthof{\ensuremath{#1}}]{\kern-.6pt\bigwedge\kern-.6pt}    {\rule[-\textheight/2]{1ex}{\textheight}}  }{\textheight}}{0.5ex}}\stackon[1pt]{#1}{\tmpbox}}
\newtheorem{theorem}{Theorem}

\newtheorem{definition}[theorem]{Definition}

\newtheorem{lemma}[theorem]{Lemma}
\newtheorem{notation}[theorem]{Notation}

\newtheorem{proposition}[theorem]{Proposition}
\newtheorem{remark}[theorem]{Remark}

\newenvironment{proof}[1][Proof]{\noindent\textbf{#1.} }{\ \rule{0.5em}{0.5em}}
\input{tcilatex}
\begin{document}

\title{Integration of geometric rough paths}
\author{Danyu Yang}
\date{}
\maketitle

\begin{abstract}
We build a connection between rough path theory and noncommutative algebra,
and interpret the integration of geometric rough paths as an example of a
non-abelian Young integration. We identify a class of slowly-varying
one-forms, and prove that the class is stable under
basic operations. In particular rough path theory is extended to allow a
natural class of time varying integrands.
\end{abstract}

\bigskip

Consider two topological groups $G_{1}$ and $G_{2}$, and a differentiable
function $f:G_{1}\rightarrow G_{2}$. For a time interval $\left[ S,T\right] $
and a differentiable path $X:\left[ S,T\right] \rightarrow G_{1}$, the
integration of the exact one-form $df$ along $X$ can be defined as:%
\begin{equation*}
\int_{r=S}^{T}dfdX_{r}=\int_{X_{S}}^{X_{T}}df=f\left( X_{S}\right)
^{-1}\!f\left( X_{T}\right) \in G_{2}\text{.}
\end{equation*}%
When $f$ and $X$ are only continuous, $\int_{r=S}^{T}dfdX_{r}$ can be
defined as $f\left( X_{S}\right) ^{-1}\!f\left( X_{T}\right) $.

Consider a time-varying exact one-form $\left( df_{t}\right) _{t}$ with $%
f_{t}:G_{1}\rightarrow G_{2}$ indexed by $t\in \left[ 0,1\right] $, and $X:%
\left[ 0,1\right] \rightarrow G_{1}$. If the following limit exists in $%
G_{2} $:%
\begin{equation*}
\lim_{\left\vert D\right\vert \rightarrow 0,D\subset \left[ 0,1\right]
}\int_{r=t_{0}}^{t_{1}}df_{t_{0}}dX_{r}%
\int_{r=t_{1}}^{t_{2}}df_{t_{1}}dX_{r}\cdots
\int_{r=t_{n-1}}^{t_{n}}df_{t_{n-1}}dX_{r}
\end{equation*}%
where $D=\left\{ t_{k}\right\} _{k=0}^{n},0=t_{0}<\cdots <t_{n}<1,n\geq 1$
with $\left\vert D\right\vert :=\max_{k=0}^{n-1}\left\vert
t_{k+1}-t_{k}\right\vert $, then the integral $\int_{r=0}^{1}df_{r}dX_{r}$
is defined to be the limit.

We reinterpret the integration of Lipschitz one-forms along geometric rough
paths developed by Lyons \cite{lyons1998differential} as an integration of
time-varying exact one-forms along group-valued paths. The interpretation is
in the language of the Malvenuto--Reutenauer Hopf algebra of permutations
introduced in \cite{malvenuto1994produits, malvenuto1995duality}. \newpage

\section{Background: rough path theory}

In \cite{young1936inequality} Young proved that, for $x:\left[ 0,1\right]
\rightarrow
\mathbb{C}
$ of finite $p$-variation and $y:\left[ 0,1\right] \rightarrow
\mathbb{C}
$ of finite $q$-variation, $p\geq 1$, $q\geq 1$, $p^{-1}+q^{-1}>1$, the
Stieltjes integral
\begin{equation*}
\int_{t=0}^{1}x_{t}dy_{t}
\end{equation*}%
is well defined\footnote{%
at least when $x$ and $y$ have no common jumps.} as the limit of Riemann
sums. The definition of $p$-variation goes back to Wiener \cite%
{wiener1924quadratic}:%
\begin{equation*}
\left\Vert x\right\Vert _{p-var,\left[ 0,1\right] }:=\sup_{D\subset \left[
0,1\right] }\left( \sum_{k,t_{k}\in D}\left\Vert
x_{t_{k+1}}-x_{t_{k}}\right\Vert ^{p}\right) ^{\frac{1}{p}}
\end{equation*}%
where the supremum is over all finite partitions $D=\left\{ t_{k}\right\}
_{k=0}^{n},0=t_{0}<\cdots <t_{n}=1,n\geq 1$. The condition given by Young is
sharp: the Riemann-Stieltjes integral $\int xdy$ does not necessarily exist
when $p^{-1}+q^{-1}=1$ \cite{young1936inequality}. In \cite{lyons1991non},
Lyons demonstrated that similar obstacle exists in stochastic integration,
and one needs to consider $x$ and $y$ as a \textquotedblleft
pair\textquotedblright\ \textquotedblleft in a fairly strong
way\textquotedblright .

From a different perspective, Chen \cite{chen1954iterated,
chen1957integration, chen1958integration} investigated the iterated
integration of one-forms and developed a theory of cohomology for loop
spaces. One of the major objects he studied is noncommutative formal series
with coefficients the iterated integrals of the coordinates of a path. For a
time interval $\left[ S,T\right] $, let $x:\left[ S,T\right] \rightarrow
\mathbb{R}
^{d}$ be a smooth path, and let $X_{1},\dots ,X_{d}$ be noncommutative
indeterminates. Consider the formal power series:%
\begin{equation}
\theta \left( x\right) :=1+\sum_{n\geq 1,i_{j}\in \left\{ 1,\dots ,d\right\}
}w_{i_{1}\cdots i_{n}}X_{i_{1}}\cdots X_{i_{n}}
\label{Chen formal power representation of paths}
\end{equation}%
where%
\begin{equation*}
w_{i_{1}\cdots i_{n}}:=\idotsint\nolimits_{S<u_{1}<\cdots
<u_{n}<T}dx_{u_{1}}^{i_{1}}\cdots dx_{u_{n}}^{i_{n}}\text{.}
\end{equation*}%
The space of paths in $%
\mathbb{R}
^{d}$ have an associative multiplication given by concatenation, and the set
of formal series have an associative multiplication that is the bilinear
extension of the concatenation of finite sequences. Chen \cite%
{chen1954iterated} proved that $\theta $ is an algebra homomorphism from
paths to formal series and satisfies $\theta \left( x\right) ^{-1}=\theta (%
\overleftarrow{x})$ where $\overleftarrow{x}$ denotes the path given by
running $x$ backwards. Based on Chen \cite{chen1957integration} $\theta $
takes values in a group whose elements are algebraic exponentials of Lie
series, and the multiplication in the group is given by
Campbell-Baker-Hausdorff formula.

In \cite{lyons1998differential}, Lyons developed the theory of rough paths.
He observed that, in controlled systems, both the driving path and the
solution path evolve in a group (denoted by $G$) rather than in a vector
space. Lyons also identified a family of metrics on group-valued paths such
that the It\^{o} map that sends a driving path to the solution path is
continuous. Elements in $G$ are algebraic exponentials of Lie series, and a
geometric $p$-rough path is a continuous path in $G$ with finite $p$%
-variation.

For a continuous bounded variation path, there exists a canonical lift of
the path to a geometric $1$-rough path given by the sequence of indefinite
iterated integrals. The sequence of definite iterated integrals is called
the signature of the continuous bounded variation path.

The following is Definition 1.1 Hambly and Lyons \cite{hambly2010uniqueness}.

\begin{definition}[Signature]
\label{Definition Signature}Let $\gamma $ be a path of bounded variation on $%
\left[ S,T\right] $ with values in a vector space $V$. Then its signature is
the sequence of definite iterated integrals%
\begin{eqnarray*}
X_{S,T} &=&\left( 1+X_{S,T}^{1}+\cdots +X_{S,T}^{k}+\cdots \right) \\
&=&\left( 1+\int_{S<u<T}d\gamma _{u}+\cdots +\int_{S<u_{1}<\cdots
<u_{k}<T}d\gamma _{u_{1}}\otimes \cdots \otimes d\gamma _{u_{k}}+\cdots
\right)
\end{eqnarray*}%
regarded as an element of an appropriate closure of the tensor algebra $%
T\left( V\right) =\bigoplus\nolimits_{n=0}^{\infty }V^{\otimes n}$.
\end{definition}

The signature is invariant under reparametrisations of the path.

Following \cite{hambly2010uniqueness} $X_{S,T}$ is also denoted by $S\left(
\gamma \right) $ where $S$ is the signature mapping that sends a continuous
bounded variation path to the sequence of definite iterated integrals.

\begin{notation}
Denote by $S:\gamma \mapsto S\left( \gamma \right) $ the signature mapping.
\end{notation}

The signature provides an efficient and effective description of the
information encoded in paths, and two paths with the same signature have the
same effect on all controlled systems. An important problem in the theory of
rough path is the `uniqueness of signature' problem: to describe the kernel
of the signature mapping. This problem goes back to Chen \cite%
{chen1958integration}\ where he proved that the map $\theta $ (signature)
provides a faithful representation for a family of paths that are
irreducible, piecewise regular, and continuous. In \cite%
{hambly2010uniqueness} Hambly and Lyons established quantitative estimates
of a path in terms of its signature, and proved the uniqueness of signature
result for continuous bounded variation paths. In \cite%
{boedihardjo2016signature} Boedihardjo, Geng, Lyons and Yang extended the
result to weak geometric rough paths over Banach spaces. There are also
progresses in the direction of machine learning, where rough path is
introduced as a new feature set for streamed data \cite{lyons2014feature,
kormilitzin2016application}.

\begin{notation}
For a vector space $V$, let $G\left( V\right) $ denote the set of group-like
elements in the tensor algebra $T\left( V\right) $.
\end{notation}

Based on Corollary 3.3 \cite{reutenauer2003free}, elements in $G\left(
V\right) $ are algebraic exponentials of Lie series, and $G\left( V\right) $
is a group.

A time interval is an interval of the form $\left[ S,T\right] $ for $0\leq
S\leq T<\infty $. For a time interval $J$ and a Banach space $V$, let $%
BV\left( J,V\right) $ denote the set of continuous bounded variation paths $%
J\rightarrow V$. Based on Chen \cite{chen1954iterated}, the signature of
paths in $BV\left( J,V\right) $ form a subgroup of $G\left( V\right) $. For $%
x\in BV\left( J_{1},V\right) $ and $y\in BV\left( J_{2},V\right) $, let $%
x\ast y$ denote the concatenation of $x$ with $y$, and let $\overleftarrow{x}
$ denote the time-reversing of $x$. Then the following identities hold:%
\begin{equation*}
S\left( x\ast y\right) =S\left( x\right) S\left( y\right) \text{ and }%
S\left( \overleftarrow{x}\right) =S\left( x\right) ^{-1}.
\end{equation*}

Let $V$ be a Banach space. Suppose the tensor powers of $V$ are equipped
with admissible norms Definition 1.25 p.20 \cite{lyons2007differential}.
Following Definition 2.1 \cite{boedihardjo2016signature}, we equip $G\left(
V\right) $ with the metric%
\begin{equation*}
d\left( a,b\right) :=\max_{k\in
\mathbb{N}
}\left\Vert \pi _{k}\left( a^{-1}b\right) \right\Vert ^{\frac{1}{k}}
\end{equation*}%
for $a,b\in G\left( V\right) $, where $\pi _{k}$ denotes the projection of $%
T\left( V\right) $ to $V^{\otimes k}$. Based on Definition 1.2.2 \cite%
{lyons1998differential} for a time interval $J$, $\omega :\left\{ \left(
s,t\right) |s\leq t,s\in J,t\in J\right\} \rightarrow \overline{%
\mathbb{R}
^{+}}\ $is a \textit{control} if $\omega $ is continuous, super-additive and
vanishes on the diagonal.

The following definition is based on Definition 1.2.2 Lyons \cite%
{lyons1998differential}.

\begin{definition}[Geometric Rough Paths]
\label{Definition geometric rough paths}For a time interval $J$ and a Banach
space $V$, $X:J\rightarrow G\left( V\right) $ is called a geometric $p$%
-rough path for some $p\geq 1$, if there exists a control $\omega :\left\{
\left( s,t\right) |s\leq t,s\in J,t\in J\right\} \rightarrow \overline{%
\mathbb{R}
^{+}}$ such that%
\begin{equation*}
\left\Vert X\right\Vert _{p-var,\left[ s,t\right] }^{p}\leq \omega \left(
s,t\right)
\end{equation*}%
for every $s\leq t$, where
\begin{equation*}
\left\Vert X\right\Vert _{p-var,\left[ s,t\right] }^{p}:=\sup_{D\subset %
\left[ s,t\right] }\sum_{k,t_{k}\in D}d\left( X_{t_{k}},X_{t_{k+1}}\right)
^{p},
\end{equation*}%
with the supremum over all $D=\left\{ t_{k}\right\} _{k=0}^{n}\subset \left[
s,t\right] ,s=t_{0}<\cdots <t_{n}=t,n\geq 1$.
\end{definition}

\section{An algebra of permutations}

\label{Section hopf algebra of permutations}

The group-like elements in the tensor algebra is the group of characters of
the shuffle algebra p.54 Theorem 3.2 Reutenauer \cite{reutenauer2003free}.
The shuffle product can be expressed in terms of permutations based on the
Malvenuto--Reutenauer Hopf algebra (denoted by MR) introduced in \cite%
{malvenuto1994produits,malvenuto1995duality}. MR is a Hopf algebra of
permutations and is noncommutative.

We first review the shuffle Hopf algebra and its dual space based on Section
1.5 Reutenauer \cite{reutenauer2003free}. Let $A$ be a (possibly infinite)
set, and let $K$ be a commutative $%
\mathbb{Q}
$-algebra. Let $K\!\left\langle A\right\rangle $ denote the set of
non-commutative polynomials on $A$ over $K$. Let $A^{\ast }$ denote the free
monoid generated by $A$ that is the set of finite sequences of elements in $%
A $ including the empty sequence denoted by $e$. The operation on $A^{\ast }$
is given by concatenation:%
\begin{equation*}
\left( a_{1}\cdots a_{n}\right) \left( a_{n+1}\cdots a_{n+m}\right)
:=a_{1}\cdots a_{n+m}
\end{equation*}%
for $a_{i}\in A$, with $e$ the identity element. There is a natural
embedding $A\hookrightarrow A^{\ast }$. Based on Ree \cite{ree1958lie} the
shuffle product $sh:K\!\left\langle A\right\rangle \otimes K\!\left\langle
A\right\rangle \rightarrow K\!\left\langle A\right\rangle $ is the $K$%
-bilinear map that can be defined recursively by%
\begin{equation*}
sh\circ \left( w_{1}a_{1}\otimes w_{2}a_{2}\right) :=\left( sh\circ \left(
w_{1}a_{1}\otimes w_{2}\right) \right) a_{2}+\left( sh\circ \left(
w_{1}\otimes w_{2}a_{2}\right) \right) a_{1}
\end{equation*}%
for $w_{i}\in A^{\ast },a_{i}\in A$, where $wa$ denotes the concatenation of
$w\in A^{\ast }$ with $a\in A$, and $sh\circ \left( e\otimes w\right)
=sh\circ \left( w\otimes e\right) :=w$ for $w\in A^{\ast }$. The product $sh$
is associative with unit $u\left( k\right) :=ke$ for $k\in K$. Let $\delta
^{\prime }:K\!\left\langle A\right\rangle \rightarrow K\!\left\langle
A\right\rangle \otimes K\!\left\langle A\right\rangle $ denote the
deconcatenation coproduct that is the $K$-linear map given by%
\begin{equation*}
\delta ^{\prime }\left( a_{1}\cdots a_{n}\right) :=\sum_{k=0}^{n}a_{1}\cdots
a_{k}\otimes a_{k+1}\cdots a_{n}
\end{equation*}%
for $a_{i}\in A,n\geq 1$, and $\delta ^{\prime }\left( e\right) :=e\otimes e$%
. The counit $\epsilon \ $is the projection of $K\!\left\langle
A\right\rangle $ to the space spanned by $e\in A^{\ast }$. $\left(
K\!\left\langle A\right\rangle ,sh,u,\delta ^{\prime },\epsilon \right) $ is
a Hopf algebra p.31 \cite{reutenauer2003free} which we call the shuffle Hopf
algebra.

Let $K\!\left\langle \!\left\langle A\right\rangle \!\right\rangle $ denote
the set of formal series $s=\sum_{w\in A^{\ast }}\left( s,w\right) w$ on $A$
over $K$. The concatenation product $conc:K\!\left\langle \!\left\langle
A\right\rangle \!\right\rangle \otimes K\!\left\langle \!\left\langle
A\right\rangle \!\right\rangle \rightarrow K\!\left\langle \!\left\langle
A\right\rangle \!\right\rangle $ is the $K$-bilinear map given by
\begin{equation*}
\left( conc\circ \left( s\otimes t\right) ,w\right) :=\sum_{uv=w}\left(
s,u\right) \left( t,v\right)
\end{equation*}%
for $w\in A^{\ast }$. The map $\delta $\ on $K\!\left\langle \!\left\langle
A\right\rangle \!\right\rangle $ is the homomorphism of the concatenation
algebra given by $\delta \left( a\right) :=e\otimes a+a\otimes e$ for each $%
a\in A$, and is of the explicit form p.25 \cite{reutenauer2003free}
\begin{equation*}
\delta \left( s\right) =\sum_{w_{1},w_{2}\in A^{\ast }}\left( s,sh\circ
\left( w_{1}\otimes w_{2}\right) \right) w_{1}\otimes w_{2}
\end{equation*}%
for $s\in K\!\left\langle \!\left\langle A\right\rangle \!\right\rangle $. $%
K\!\left\langle \!\left\langle A\right\rangle \!\right\rangle $ is nearly a
Hopf algebra, but the map $\delta $ does not necessarily take values in $%
K\!\left\langle \!\left\langle A\right\rangle \!\right\rangle \otimes
K\!\left\langle \!\left\langle A\right\rangle \!\right\rangle $ and the sum
in $\delta \left( s\right) =\sum s_{\left( 1\right) }\otimes s_{\left(
2\right) }$ can be infinite p.38 \cite{reutenauer2003free}.

There is a duality between $K\!\left\langle \!\left\langle A\right\rangle
\!\right\rangle $ and $K\!\left\langle A\right\rangle $ given by%
\begin{equation*}
\left( s,p\right) :=\sum_{w\in A^{\ast }}\left( s,w\right) \left( p,w\right)
\end{equation*}%
for $s\in K\!\left\langle \!\left\langle A\right\rangle \!\right\rangle $
and $p\in K\!\left\langle A\right\rangle $. Then the following duality holds
p.26 \cite{reutenauer2003free}:%
\begin{eqnarray*}
\left( s,sh\circ \left( p\otimes q\right) \right) &=&\left( \delta
s,p\otimes q\right) \\
\left( conc\circ \left( s\otimes t\right) ,p\right) &=&\left( s\otimes
t,\delta ^{\prime }p\right)
\end{eqnarray*}%
for $s,t\in K\!\left\langle \!\left\langle A\right\rangle \!\right\rangle $
and $p,q\in K\!\left\langle A\right\rangle $.

\begin{definition}[Group-like Elements]
The group-like elements in $K\!\left\langle \!\left\langle A\right\rangle
\!\right\rangle $ is the set of\ $s\in K\!\left\langle \!\left\langle
A\right\rangle \!\right\rangle $ that satisfies $\delta s=s\otimes s$.
\end{definition}

\begin{notation}
Let $G\left( A\right) $ denote the set of group-like elements in $%
K\!\left\langle \!\left\langle A\right\rangle \!\right\rangle $.
\end{notation}

Based on Theorem 3.2 p.54 \cite{reutenauer2003free} the group-like elements
in $K\!\left\langle \!\left\langle A\right\rangle \!\right\rangle $ are
algebraic exponentials of Lie series, and they form a group in $%
K\!\left\langle \!\left\langle A\right\rangle \!\right\rangle $ Corollary
3.3 \cite{reutenauer2003free}. Based on the duality between $sh$ and $\delta
$, group-like elements in $K\!\left\langle \!\left\langle A\right\rangle
\!\right\rangle $ are the set of characters of the shuffle algebra on $%
K\!\left\langle A\right\rangle $ Theorem 3.2 \cite{reutenauer2003free}: $%
s\in G\left( A\right) $ iff $\left( s,p\right) \left( s,q\right) =\left(
s,sh\circ \left( p\otimes q\right) \right) $ for every $p,q\in
K\!\left\langle A\right\rangle $.

MR is a $%
\mathbb{Z}
$-Hopf algebra on permutations $S:=\cup _{n\geq 0}S_{n}$.

Let $\limfunc{End}\left( K\!\left\langle A\right\rangle \right) $ denote the
$K$-module of linear endomorphisms of $K\!\left\langle A\right\rangle $.
Based on Proposition 1.10 \cite{reutenauer2003free} $\limfunc{End}\left(
K\!\left\langle A\right\rangle \right) $ becomes a $K$-associative algebra
with the convolution product $\ast ^{\prime }$ given by
\begin{equation*}
f\ast ^{\prime }g=sh\circ \left( f\otimes g\right) \circ \delta ^{\prime }
\end{equation*}%
for $f,g\in \limfunc{End}\left( K\!\left\langle A\right\rangle \right) $.
There is an embedding of permutations $%
\mathbb{Z}
S$ in $\limfunc{End}\left( K\!\left\langle A\right\rangle \right) $ given by%
\begin{equation*}
\sigma \cdot \left( a_{1}\cdots a_{n}\right) :=a_{\sigma 1}\cdots a_{\sigma
n}
\end{equation*}%
for $\sigma \in S_{n}$ and $a_{i}\in A$. Based on \cite{malvenuto1995duality}%
, the product in MR is the $%
\mathbb{Z}
$-bilinear map $\ast ^{\prime }:%
\mathbb{Z}
S\times
\mathbb{Z}
S\rightarrow
\mathbb{Z}
S$ and is of the explicit form%
\begin{equation*}
\sigma \ast ^{\prime }\rho :=sh\circ \left( \sigma \otimes \bar{\rho}\right)
\end{equation*}%
for $\sigma \in S_{n}$ and $\rho \in S_{m}$, where permutations are
considered as words with $\bar{\rho}\left( i\right) :=n+\rho \left( i\right)
$. The product $\ast ^{\prime }$ is associative with identity element $%
\lambda \in S_{0}$. The coproduct on MR $\bigtriangleup ^{\prime }:%
\mathbb{Z}
S\rightarrow
\mathbb{Z}
S\otimes
\mathbb{Z}
S$ is the $%
\mathbb{Z}
$-linear map given by%
\begin{equation*}
\bigtriangleup ^{\prime }:=\left( \limfunc{st}\otimes \limfunc{st}\right)
\circ \delta ^{\prime }
\end{equation*}%
where `$\limfunc{st}$' denotes the unique increasing map that sends a
sequence of $k$ non-repeating integers to $\left\{ 1,2,\dots ,k\right\} $
for $k\geq 1$. The counit is the projection of $%
\mathbb{Z}
S$ to the space spanned by $\lambda \in S_{0}$. Based on \cite%
{malvenuto1994produits, malvenuto1995duality, poirier1995algebres} MR\textbf{%
\ }is a Hopf algebra that is self-dual, free and cofree, so is neither
commutative nor cocommutative.

For $f\in \limfunc{End}\left( K\!\left\langle A\right\rangle \right) $,
define the adjoint map $f^{\ast }\in \limfunc{End}\left( K\!\left\langle
\!\left\langle A\right\rangle \!\right\rangle \right) $ as%
\begin{equation*}
\left( f^{\ast }s,p\right) :=\left( s,fp\right)
\end{equation*}%
for $s\in K\!\left\langle \!\left\langle A\right\rangle \!\right\rangle $
and $p\in K\!\left\langle A\right\rangle $. As a sub-algebra of $\limfunc{End%
}\left( K\!\left\langle A\right\rangle \right) $, MR induces a sub-algebra
of $\limfunc{End}\left( K\!\left\langle \!\left\langle A\right\rangle
\!\right\rangle \right) $. Proposition \ref{Proposition free Lie group is a
group of characters} below states that the group-like elements in $%
K\!\left\langle \!\left\langle A\right\rangle \!\right\rangle $ (denoted by $%
G\left( A\right) $) is a group of characters of MR: for $s\in G\left(
A\right) $,
\begin{eqnarray*}
\widehat{s}:\left(
\mathbb{Z}
S,\ast ^{\prime }\right) &\rightarrow &\left( K\!\left\langle \!\left\langle
A\right\rangle \!\right\rangle ,conc\right) \\
\sigma &\mapsto &\sigma ^{\ast }s
\end{eqnarray*}%
is an algebra homomorphism.\ Proposition \ref{Proposition free Lie group is
a group of characters} is closely related to the cotensor algebra p.248
Ronco \cite{ronco2000primitive}.

\begin{proposition}
\label{Proposition free Lie group is a group of characters}For $s\in G\left(
A\right) $, define a $%
\mathbb{Z}
$-linear map $\widehat{s}:%
\mathbb{Z}
S\rightarrow K\!\left\langle \!\left\langle A\right\rangle \!\right\rangle $
by%
\begin{equation*}
\widehat{s}\left( \sigma \right) :=\sum_{w\in A^{\ast }}\left( s,\sigma
\cdot w\right) w.
\end{equation*}%
Then%
\begin{equation*}
conc\circ \left( \widehat{s}\left( \sigma \right) \otimes \widehat{s}\left(
\rho \right) \right) =\widehat{s}\left( \sigma \ast ^{\prime }\rho \right)
\end{equation*}%
for $\sigma ,\rho \in
\mathbb{Z}
S$.
\end{proposition}

\begin{proof}
Since $\sigma \ast ^{\prime }\rho :=sh\circ \left( \sigma \otimes \rho
\right) \circ \delta ^{\prime }$,%
\begin{equation*}
\left( \sigma \ast ^{\prime }\rho \right) \cdot w=sh\circ \left( \left(
\sigma \cdot u\right) \otimes \left( \rho \cdot v\right) \right)
\end{equation*}%
for $\sigma \in S_{n},\rho \in S_{m},w\in A^{\ast },w=uv,\left\vert
u\right\vert =n,\left\vert v\right\vert =m$, where $\left\vert u\right\vert $
denotes the number of letters in $u\in A^{\ast }$. Since $s\in G\left(
A\right) $ is a character of the shuffle algebra on $K\!\left\langle
A\right\rangle $,%
\begin{equation*}
\left( s,\left( \sigma \ast ^{\prime }\rho \right) \cdot w\right) =\left(
s,sh\circ \left( \left( \sigma \cdot u\right) \otimes \left( \rho \cdot
v\right) \right) \right) =\left( s,\sigma \cdot u\right) \left( s,\rho \cdot
v\right) .
\end{equation*}%
The statement follows.
\end{proof}

Based on the bidendriform algebra structure of the shuffle algebra Loday and
Ronco \cite{loday1998hopf} Loday \cite{loday2001dialgebras}, consider $\succ
$ that is an abstract iterated integration.

\begin{notation}
Denote $\succ :\left( K\!\left\langle A\right\rangle \times K\!\left\langle
A\right\rangle \right) \backslash \left( K\!e\times K\!e\right) \rightarrow
K\!\left\langle A\right\rangle $ as the $K$-bilinear map given by%
\begin{align*}
& \left( a_{1}\cdots a_{n}\right) \succ \left( a_{n+1}\cdots a_{n+m}\right)
\\
:=\text{ }& sh\circ \left( \left( a_{1}\cdots a_{n}\right) \otimes \left(
a_{n+1}\cdots a_{n+m-1}\right) \right) a_{n+m}
\end{align*}%
for $a_{i}\in A,m\geq 1$, where $wa$ denotes the concatenation of $w\in
A^{\ast }$ with $a\in A$; $\left( a_{1}\cdots a_{n}\right) \succ e:=0\in K$
for $a_{i}\in A,n\geq 1$.
\end{notation}

In defining the integration of geometric rough paths, Lyons considered an
ordered shuffle to define almost multiplicative functionals p.285 Definition
3.2.2 \cite{lyons1998differential}. The ordered shuffle can be viewed as
iterated applications of $\succ $ as defined in Notation 2.3 \cite%
{ronco2000primitive}.

The following Lemma helps to prove that indefinite integrals of one-forms
along geometric rough paths are geometric rough paths. Consider $p_{1},\dots
,p_{n}\in K\!\left\langle A\right\rangle \ $that satisfy $\left(
p_{i},e\right) =0,i=1,\dots ,n$, where $p=\sum_{w\in A^{\ast }}\left(
p,w\right) w$ and $e$ is the empty sequence in $A^{\ast }$. Define%
\begin{eqnarray*}
m_{\succ }\left( p_{1}\right) &:&=p_{1} \\
m_{\succ }\left( p_{1},\cdots ,p_{n}\right) &:&=\left( \cdots \left(
p_{1}\succ p_{2}\right) \succ \cdots \succ p_{n-1}\right) \succ p_{n}.
\end{eqnarray*}

\begin{lemma}
\label{Proposition iterated integrals of words}Let $p_{1},\dots ,p_{n+m}\in
K\!\left\langle A\right\rangle \ $satisfy $\left( p_{i},e\right)
=0,i=1,\dots ,n+m$. Then%
\begin{eqnarray*}
&&sh\circ \left( m_{\succ }\left( p_{1},\cdots ,p_{n}\right) \otimes
m_{\succ }\left( p_{n+1},\cdots ,p_{n+m}\right) \right) \\
&=&\sum_{\rho \in 1_{n}\ast ^{\prime }1_{m}}m_{\succ }\left( p_{\rho \left(
1\right) },\cdots ,p_{\rho \left( n+m\right) }\right) .
\end{eqnarray*}
\end{lemma}

\begin{proof}
Based on the multi-linearity of $m_{\succ }\left( ,\cdots ,\right) $, we
assume $p_{i}=w_{i}$ for $w_{i}\in A^{\ast },\left\vert w_{i}\right\vert
\geq 1$, $i=1,\dots ,n+m$. Based on the order of the image of the last
element of $\left( w_{i}\right) _{i=1}^{n}$ in $m_{\succ }\left(
w_{1},\cdots ,w_{n}\right) $, elements in $m_{\succ }\left( w_{1},\cdots
,w_{n}\right) $ can be divided into $\frac{\left( n+m\right) !}{n!m!}$
subsets indexed by elements in $1_{n}\ast ^{\prime }1_{m}$. Indeed, denote $%
l_{0}:=0,l_{k}:=\sum_{j=1}^{k}\left\vert w_{i}\right\vert ,k=1,\dots ,n+m$.
For each $\rho \in 1_{n}\ast ^{\prime }1_{m}$, let $E\left( \rho \right) $
denote the set of $\alpha \in S_{l_{n+m}}$ that satisfy $\alpha \left(
l_{k}+i\right) <\alpha \left( l_{k}+i+1\right) ,i=1,\dots ,\left\vert
w_{k+1}\right\vert -1,k=0,\dots ,n+m-1$ and $\alpha \left( l_{\rho \left(
k\right) }\right) <\alpha \left( l_{\rho \left( k+1\right) }\right)
,k=1,\dots ,n+m-1$. Then
\begin{eqnarray*}
&&sh\circ \left( m_{\succ }\left( w_{1},\cdots ,w_{n}\right) \otimes
m_{\succ }\left( w_{n+1},\cdots ,w_{n+m}\right) \right) \\
&=&\sum_{\rho \in 1_{n}\ast ^{\prime }1_{m}}\sum_{\alpha \in E\left( \rho
\right) }\alpha \cdot \left( w_{1}\cdots w_{n+m}\right) \\
&=&\sum_{\rho \in 1_{n}\ast ^{\prime }1_{m}}m_{\succ }\left( w_{\rho \left(
1\right) },\cdots ,w_{\rho \left( n+m\right) }\right)
\end{eqnarray*}%
where in the second line $w_{1}\cdots w_{n+m}\in A^{\ast }$ denotes the
concatenation of $w_{i}$.
\end{proof}

The operation $\succ :K\!\left\langle A\right\rangle \times K\!\left\langle
A\right\rangle \rightarrow K\!\left\langle A\right\rangle $ induces an
operation on $\limfunc{End}\left( K\!\left\langle A\right\rangle \right) $
given by%
\begin{equation*}
f\succ g:=\succ \circ \left( f\otimes g\right) \circ \delta ^{\prime }.
\end{equation*}%
$\succ $ is closed on permutations \cite{ronco2000primitive}.

\begin{notation}
Let $\succ :\left(
\mathbb{Z}
S\times
\mathbb{Z}
S\right) \backslash \left(
\mathbb{Z}
S_{0}\times
\mathbb{Z}
S_{0}\right) \rightarrow
\mathbb{Z}
S$ denote the $%
\mathbb{Z}
$-bilinear map given by%
\begin{equation*}
\sigma \succ \rho :=\sigma \succ \bar{\rho}
\end{equation*}%
for $\sigma \in S_{n},\rho \in S_{m}$, where permutations are viewed as
words with $\bar{\rho}\left( i\right) :=n+\rho \left( i\right) $.
\end{notation}

MR is actually a bidendriform algebra Loday \cite{loday2001dialgebras}. With
$\left( 1\right) \in S_{1}$, define%
\begin{eqnarray*}
\mathcal{I}:%
\mathbb{Z}
S &\rightarrow &%
\mathbb{Z}
S \\
\sigma &\mapsto &\sigma \succ \left( 1\right) .
\end{eqnarray*}%
$\mathcal{I}$ can be viewed as an abstract integration. Since $%
\bigtriangleup ^{\prime }\left( 1\right) =\left( 1\right) \otimes \lambda
+\lambda \otimes \left( 1\right) $ for $\lambda \in S_{0}$, based on p.248
Definition 1.2 (b) Ronco \cite{ronco2000primitive}, for $\sigma \in
\mathbb{Z}
S$,
\begin{equation*}
\bigtriangleup ^{\prime }\mathcal{I}\left( \sigma \right) \mathcal{=I}\left(
\sigma \right) \otimes \lambda +\sum_{\bigtriangleup ^{\prime }\sigma
=\sigma _{\left( 1\right) }\otimes \sigma _{\left( 2\right) }}\sigma
_{\left( 1\right) }\otimes \mathcal{I}\left( \sigma _{\left( 2\right)
}\right) .
\end{equation*}

Proposition \ref{Proposition expression of iterated integration of
permutations} below helps to prove that slowly-varying one-forms are closed
under iterated integration (Proposition \ref{Proposition stability under
iterated integration}). Let $1_{n}$ denote the identity element in $%
S_{n},n\geq 1$, and $1_{0}:=\lambda \in S_{0}$. $G\left( A\right) $ denotes
the group-like elements in $K\!\left\langle \!\left\langle A\right\rangle
\!\right\rangle $.

\begin{proposition}
\label{Proposition expression of iterated integration of permutations}For $%
s\in G\left( A\right) $, define $\widehat{s}:%
\mathbb{Z}
S\rightarrow K\!\left\langle \!\left\langle A\right\rangle \!\right\rangle $
as in Proposition \ref{Proposition free Lie group is a group of characters}.
Then%
\begin{eqnarray*}
&&\savestack{\tmpbox}{\stretchto{ \scaleto{
\scalerel*[\widthof{\ensuremath{conc\circ \left( s\otimes t\right)
}}]{\kern-.6pt\bigwedge\kern-.6pt} {\rule[-\textheight/2]{1ex}{\textheight}}
}{\textheight}}{0.5ex}}\stackon[1pt]{conc\circ \left( s\otimes t\right)
}{\tmpbox}\left( 1_{n_{1}}\succ 1_{n_{2}}\right) -\widehat{s}\left(
1_{n_{1}}\succ 1_{n_{2}}\right) \\
&=&\sum_{\substack{ k_{1}=0,\dots ,n_{1}  \\ k_{2}=0,\dots ,n_{2}-1}}\rho
_{k_{2},n_{1}-k_{1}}\cdot \left( conc\circ \left( \widehat{s}\left(
1_{k_{1}}\ast ^{\prime }1_{k_{2}}\right) \otimes \widehat{t}\left(
1_{n_{1}-k_{1}}\succ 1_{n_{2}-k_{2}}\right) \right) \right)
\end{eqnarray*}%
for $s,t\in G\left( A\right) $ and $n_{1}\geq 0,n_{2}\geq 1$, where $\sigma
\cdot \left( a_{1}\cdots a_{n}\right) :=a_{\sigma 1}\cdots a_{\sigma n}$ for
$\sigma \in S_{n}$ and $a_{i}\in A$; $\rho _{k_{2},n_{1}-k_{1}}\in
S_{n_{1}+n_{2}}$ is given by changing the order of two sub-sequences $\left(
k_{1}+1,\dots ,k_{1}+k_{2}\right) $ and $\left( k_{1}+k_{2}+1,\dots
,n_{1}+k_{2}\right) $ in $\left( 1,\dots ,n_{1}+n_{2}\right) $.
\end{proposition}

\begin{remark}
The equality in Proposition \ref{Proposition expression of iterated
integration of permutations} can be formally expressed as%
\begin{equation*}
\left( \int x^{n_{1}}dx^{n_{2}}\right) \bigg|_{x=s}^{st}=\left( \int \left(
sx\right) ^{n_{1}}d\left( sx\right) ^{n_{2}}\right) \bigg|_{x=1}^{t}
\end{equation*}%
and is a change of variable formula for abstract iterated integration.
\end{remark}

\begin{proof}
For $w\in A^{\ast }$ let $\left\vert w\right\vert $ denote the number of
letters in $w$. Based on the bidendriform algebra structure of shuffle
algebra and Definition 1.2 (b) \cite{ronco2000primitive},%
\begin{equation*}
\delta ^{\prime }\left( u\succ v\right) =\left( u\succ v\right) \otimes
e+\sum_{\left\vert v_{\left( 2\right) }\right\vert \geq 1}\left( u_{\left(
1\right) }\shuffle v_{\left( 1\right) }\right) \otimes \left( u_{\left(
2\right) }\succ v_{\left( 2\right) }\right)
\end{equation*}%
for $u,v\in A^{\ast },\left\vert v\right\vert \geq 1$, where $u_{\left(
1\right) }\otimes u_{\left( 2\right) }:=\delta ^{\prime }\left( u\right) $, $%
v_{\left( 1\right) }\otimes v_{\left( 2\right) }:=\delta ^{\prime }\left(
v\right) $ and $\shuffle$ denotes the shuffle product. Then based on the
definition $1_{n_{1}}\succ 1_{n_{2}}:=\succ \circ \left( 1_{n_{1}}\otimes
1_{n_{2}}\right) \circ \delta ^{\prime }$, for $w\in A^{\ast
},w=uv,\left\vert u\right\vert =n_{1},\left\vert v\right\vert =n_{2}$,
\begin{eqnarray*}
&&\left( conc\circ \left( s\otimes t\right) ,\left( 1_{n_{1}}\succ
1_{n_{2}}\right) \cdot w\right) \\
&=&\left( conc\circ \left( s\otimes t\right) ,u\succ v\right) \\
&=&\left( s\otimes t,\delta ^{\prime }\left( u\succ v\right) \right) \text{
(duality between }conc\text{ and }\delta ^{\prime }\text{)} \\
&=&\left( s,u\succ v\right) +\sum_{\left\vert v_{\left( 2\right)
}\right\vert \geq 1}\left( s,u_{\left( 1\right) }\shuffle v_{\left( 1\right)
}\right) \left( t,u_{\left( 2\right) }\succ v_{\left( 2\right) }\right) \\
&=&\left( s,\left( 1_{n_{1}}\succ 1_{n_{2}}\right) \cdot w\right) \\
&&+\sum_{\left\vert v_{\left( 2\right) }\right\vert \geq 1}\left( s,\left(
1_{\left\vert u_{\left( 1\right) }\right\vert }\ast ^{\prime }1_{\left\vert
v_{\left( 1\right) }\right\vert }\right) \cdot \left( u_{\left( 1\right)
}v_{\left( 1\right) }\right) \right) \left( t,\left( 1_{\left\vert u_{\left(
2\right) }\right\vert }\succ 1_{\left\vert v_{\left( 2\right) }\right\vert
}\right) \cdot \left( u_{\left( 2\right) }v_{\left( 2\right) }\right) \right)
\end{eqnarray*}%
for $s,t\in G\left( A\right) $.
\end{proof}

Although expressed in terms of shuffles/permutations, core arguments in this
section (in particular Proposition \ref{Proposition expression of iterated
integration of permutations}) can be applied to a general bidendriform
algebra.

\section{Integration of geometric rough paths}

We first consider an example that is simple and important.

For two Banach spaces $V$ and $U$, let $L\left( V,U\right) $ denote the set
of continuous linear mappings from $V$ to $U$. Consider a polynomial
one-form $p:V\rightarrow L\left( V,U\right) $ that is a polynomial taking
values in $L\left( V,U\right) $. For $v,w,v_{0}\in V$,
\begin{equation*}
p\left( v\right) \left( w\right) =\sum_{k=0}^{n}\left( D^{k}p\right) \left(
v_{0}\right) \frac{\left( v-v_{0}\right) ^{\otimes k}}{k!}\left( w\right)
\text{,}
\end{equation*}%
where $p\left( v\right) \in L\left( V,U\right) $ and $p\left( v\right)
\left( w\right) \in U$. The value of $p\left( v\right) \left( w\right) $
does not depend on $v_{0}$.

For a time interval $\left[ S,T\right] $ and Banach space $V$, let $BV\left( %
\left[ S,T\right] ,V\right) $ denote the set of continuous bounded variation
paths $\left[ S,T\right] \rightarrow V$.

Let $x\in BV\left( \left[ S,T\right] ,V\right) $ satisfy $x_{S}=0$. Then%
\begin{eqnarray*}
&&\int_{r=S}^{T}p\left( x_{r}\right) dx_{r} \\
&=&\sum_{k=0}^{n}\left( D^{k}p\right) \left( 0\right) \int_{r=S}^{T}\frac{%
\left( x_{r}\right) ^{\otimes k}}{k!}\otimes dx_{r} \\
&=&\sum_{k=0}^{n}\left( D^{k}p\right) \left( 0\right) \int_{S<u_{1}<\cdots
<u_{k+1}<T}dx_{u_{1}}\otimes \cdots \otimes dx_{u_{k+1}} \\
&=&:\sum_{k=0}^{n}\left( D^{k}p\right) \left( 0\right) X_{S,T}^{k+1}
\end{eqnarray*}%
where the first equality is the Taylor expansion of $p$ and the second
equality is based on the symmetry of $D^{k}p$. As a result, the classical
integral $\int p\left( x\right) dx$ is expressed as a finite linear
combination of iterated integrals of $x$.

Let $G\left( V\right) $ denote the set of group-like elements in the tensor
algebra $T\left( V\right) $. Based on Chen \cite{chen1954iterated}, the
signature of continuous bounded variation paths form a subgroup of $G\left(
V\right) $.

\begin{notation}
For a polynomial one-form $p:V\rightarrow L\left( V,U\right) $ of degree $n$%
, define $f_{p}:G\left( V\right) \rightarrow U$ as%
\begin{equation}
f_{p}\left( g\right) :=\sum_{k=0}^{n}\left( D^{k}p\right) \left( 0\right)
g^{k+1}  \label{definition of fp}
\end{equation}%
where$\ g=\sum_{k\geq 0}g^{k}$ with $g^{k}\in V^{\otimes k}$.
\end{notation}

A polynomial one-form $p:V\rightarrow L\left( V,U\right) $ can be lifted to
an exact one-form $df_{p}$ for $f_{p}:G\left( V\right) \rightarrow U$
defined at $\left( \ref{definition of fp}\right) $. Indeed, for $x\in
BV\left( \left[ S,T\right] ,V\right) $, if denote $X_{t}:=\exp \left(
x_{S}\right) S\left( x|_{\left[ S,t\right] }\right) $ for each $t$, then
\begin{equation*}
\int_{r=S}^{T}p\left( x_{r}\right) dx_{r}=f_{p}\left( X_{T}\right)
-f_{p}\left( X_{S}\right) \text{.}
\end{equation*}%
When $x_{S}=0$, $f\left( X_{S}\right) =0$ and the equality holds based on
the calculation above. When $x_{S}\neq 0$, by connecting $0$ and $x_{S}$
with a straight line and using the additive property of integrals, the
equality $\int_{r=S}^{T}p\left( x_{r}\right) dx_{r}=f_{p}\left( X_{T}\right)
-f_{p}\left( X_{S}\right) $ still holds. Then based on the fundamental
theorem of calculus and the change of variable formula
\begin{equation*}
f_{p}\left( X_{T}\right) -f_{p}\left( X_{S}\right)
=\int_{X_{S}}^{X_{T}}df_{p}=\int_{r=S}^{T}df_{p}dX_{r}.
\end{equation*}%
As a result, the polynomial one-form $p:V\rightarrow L\left( V,U\right) $ is
lifted to an exact one-form $df_{p}$ for $f_{p}:G\left( V\right) \rightarrow
U$ such that%
\begin{equation*}
\int_{r=S}^{T}p\left( x_{r}\right) dx_{r}=\int_{r=S}^{T}df_{p}dX_{r}
\end{equation*}%
for each $x\in BV\left( \left[ S,T\right] ,V\right) $, where $X_{t}:=\exp
\left( x_{S}\right) S\left( x|_{\left[ S,t\right] }\right) $.

The lifting of a path to a rough path simplifies the formulation of
integration, and is necessary. A rough path can be viewed as a basis of
controlled systems, and integration/differential equation can be viewed as a
transformation between bases of controlled systems. The metric on rough path
space can be considerably weaker ($p$-variation, $p<\infty $) than the
metric needed to define iterated integrals ($p$-variation, $p<2$). When the
metric is weaker than $p<2$, the basis systems $1,\dots ,\left[ p\right] $
are selected to postulate iterated integrals and satisfy an abstract
`integration by parts formula' (defines a character of shuffle algebra for
each fixed time). The algebraic structure is important to interpret the
limit behavior of controlled systems. For example, physical Brownian motion
in a magnetic field can be described by Brownian motion with a
`non-canonical' L\'{e}vy area \cite{friz2015physical}.

Consider the lifting of the classical integral $\int p\left( x\right) dx$ to
$f_{p}:G\left( V\right) \rightarrow U$. The function $f_{p}$ takes values in
the Banach space $U$ same as the integral $\int p\left( x\right) dx$. The
full rough integral is a mapping between paths in groups, and we need to
lift $f_{p}:G\left( V\right) \rightarrow U$ to a function $F_{p}:G\left(
V\right) \rightarrow G\left( U\right) $ such that%
\begin{equation*}
S\left( \int_{r=S}^{\cdot }p\left( x_{r}\right) dx_{r}\right) =F_{p}\left(
X_{S}\right) ^{-1}F_{p}\left( X_{T}\right)
\end{equation*}%
for each $x\in BV\left( \left[ S,T\right] ,V\right) $, where $%
\int_{r=S}^{\cdot }p\left( x_{r}\right) dx_{r}\in BV\left( \left[ S,T\right]
,U\right) $ denotes the integral path and $X_{t}:=\exp \left( x_{S}\right)
S\left( x|_{\left[ S,t\right] }\right) $. Lyons defined the lift of $f_{p}$
to $F_{p}$ p.285 Definition 3.2.2 \cite{lyons1998differential}. The lift can
also be interpreted as iterated integrals of controlled paths p.101 Theorem
1 Gubinelli \cite{gubinelli2004controlling}.

We interpret the lifting of $f_{p}$ to $F_{p}$ in the language of the
Malvenuto--Reutenauer Hopf algebra of permutations (denoted by MR) \cite%
{malvenuto1994produits, malvenuto1995duality}. Let $S_{n}$ denote the
symmetric group of order $n$ for $n\geq 1$, and $S_{0}=\left\{ \lambda
\right\} $. MR is a Hopf algebra on $%
\mathbb{Z}
S$ with\ $S:=\cup _{n=0}^{\infty }S_{n}$. The product on MR $\ast ^{\prime }:%
\mathbb{Z}
S\times
\mathbb{Z}
S\rightarrow
\mathbb{Z}
S$ is the $%
\mathbb{Z}
$-bilinear map given by%
\begin{equation*}
\sigma \ast ^{\prime }\rho :=\sigma \shuffle\bar{\rho}
\end{equation*}%
for $\sigma \in S_{n}$ and $\rho \in S_{m}$, where permutations are
considered as words with $\bar{\rho}\left( i\right) :=n+\rho \left( i\right)
$ and $\shuffle$ denotes the shuffle product. For example,
\begin{equation*}
\left( 1\right) \ast ^{\prime }\left( 21\right) =1\shuffle32=\left(
132\right) +\left( 312\right) +\left( 321\right) .
\end{equation*}

Based on Ronco \cite{ronco2000primitive}, there is a natural operation $%
\succ :\left(
\mathbb{Z}
S\times
\mathbb{Z}
S\right) \backslash \left(
\mathbb{Z}
S_{0}\times
\mathbb{Z}
S_{0}\right) \rightarrow
\mathbb{Z}
S$ that is the $%
\mathbb{Z}
$-bilinear map given by%
\begin{equation*}
\sigma \succ \rho :=\left( \sigma \shuffle\bar{\rho}\left( \left[ m-1\right]
\right) \right) \bar{\rho}\left( m\right)
\end{equation*}%
for $\sigma \in S_{n},\rho \in S_{m},m\geq 1$, where permutations are
considered as words with $\bar{\rho}\left( i\right) :=n+\rho \left( i\right)
$, $\bar{\rho}\left( \left[ m-1\right] \right) $ denotes the word consisting
the first $m-1$ letters in $\bar{\rho}$, and $wi$ denotes the concatenation
of $w$ with $i$. For example,%
\begin{equation*}
\left( 1\right) \succ \left( 312\right) =\left( 1\shuffle42\right) 3=\left(
1423\right) +\left( 4123\right) +\left( 4213\right) .
\end{equation*}%
Set $\sigma \succ \lambda :=0\in
\mathbb{Z}
$ for $\sigma \in S_{n},n\geq 1,\lambda \in S_{0}$.

The operation $\succ $ can be viewed as an abstract iterated integration.
For $\rho _{i}\in
\mathbb{Z}
S,i=1,\dots ,n$, denote p.252 Notation 2.3 \cite{ronco2000primitive}
\begin{equation*}
m_{\succ }\left( \rho _{1},\cdots ,\rho _{n}\right) :=\left( \cdots \left(
\left( \rho _{1}\succ \rho _{2}\right) \succ \rho _{3}\right) \cdots \right)
\succ \rho _{n}\text{.}
\end{equation*}

For a (possibly infinite) set $A$ and a commutative $%
\mathbb{Q}
$-algebra $K$, let $K\!\left\langle \!\left\langle A\right\rangle
\!\right\rangle $ resp. $K\!\left\langle A\right\rangle $ denote the ring of
non-commutative formal series resp. polynomials on $A$ over $K$. There is a
bilinear action of $%
\mathbb{Z}
S$ on $K\!\left\langle A\right\rangle $ given by%
\begin{equation*}
\sigma \cdot a_{1}\cdots a_{n}:=a_{\sigma 1}\cdots a_{\sigma n}
\end{equation*}%
for $a_{i}\in A,\sigma \in S_{n}$. Let $A^{\ast }$ denote the free monoid
generated by $A$ that is the set of sequences $a_{1}\cdots a_{n}$ of
elements in $A$ including the empty sequence $e$ with the operation of
concatenation. Let $G\left( A\right) $ denote the set of group-like elements
in $K\!\left\langle \!\left\langle A\right\rangle \!\right\rangle $ that is
the set of characters of the shuffle algebra on $K\!\left\langle
A\right\rangle $: $s\in G\left( A\right) $ iff $\left( s,p\shuffle q\right)
=\left( s,p\right) \left( s,q\right) $ for $p,q\in K\!\left\langle
A\right\rangle $. For $s\in G\left( A\right) $, define $\widehat{s}:%
\mathbb{Z}
S\rightarrow K\!\left\langle \!\left\langle A\right\rangle \!\right\rangle $
as%
\begin{equation}
\widehat{s}\left( \sigma \right) :=\sum_{w\in A^{\ast }}\left( s,\sigma
\cdot w\right) w  \label{Definition of s hat}
\end{equation}%
for $\sigma \in
\mathbb{Z}
S$. Based on Proposition \ref{Proposition free Lie group is a group of
characters},
\begin{equation*}
conc\circ \left( \widehat{s}\left( \sigma \right) \otimes \widehat{s}\left(
\rho \right) \right) =\widehat{s}\left( \sigma \ast ^{\prime }\rho \right)
\end{equation*}%
for $\sigma ,\rho \in
\mathbb{Z}
S$, where $conc$ denotes the concatenation product.

Let $V$ and $U$ be two Banach spaces. For a polynomial one-form $%
p:V\rightarrow L\left( V,U\right) $ of degree $n$, let $\left( D^{k}p\right)
\left( 0\right) \in L\left( V^{\otimes \left( k+1\right) },U\right) $ denote
the $k$th derivative of $p$ evaluated at $0\in V$. Let $1_{k}$ denote the
identity element in $S_{k}$, $k\geq 1$. For $l=1,2,\dots $, denote%
\begin{equation*}
\sigma _{l}:=\sum_{\substack{ k_{i}=0,\dots ,n  \\ i=1,\dots ,l}}\left(
D^{k_{1}}p\right) \left( 0\right) \otimes \cdots \otimes \left(
D^{k_{l}}p\right) \left( 0\right) m_{\succ }\left( 1_{k_{1}+1},\cdots
,1_{k_{l}+1}\right) \text{.}
\end{equation*}%
Let $G\left( V\right) $ be the set of group-like elements in the tensor
algebra $T\left( V\right) $. For simplicity, we assume that $V$ has a
(possibly infinite) basis given by a set $A$, and let $G(V)$ be the set of
group-like elements in $K\!\left\langle \!\left\langle A\right\rangle
\!\right\rangle $.

\begin{notation}
For a polynomial one-form $p:V\rightarrow L\left( V,U\right) $ of degree $n$%
, define $F_{p}:G\left( V\right) \rightarrow T\left( U\right) $ as%
\begin{equation}
F_{p}\left( s\right) :=1+\sum_{l=1}^{\infty }\widetilde{s}\left( \sigma
_{l}\right) ,s\in G\left( V\right) ,  \label{Definition of F}
\end{equation}%
where $\widetilde{s}\left( \sigma _{l}\right) \in U^{\otimes l}$ is given by
\begin{equation*}
\widetilde{s}\left( \sigma _{l}\right) :=\sum_{\substack{ k_{i}=0,\dots ,n
\\ i=1,\dots ,l}}\left( D^{k_{1}}p\right) \left( 0\right) \otimes \cdots
\otimes \left( D^{k_{l}}p\right) \left( 0\right) \widehat{s}\left( m_{\succ
}\left( 1_{k_{1}+1},\cdots ,1_{k_{l}+1}\right) \right)
\end{equation*}%
with $\widehat{s}$ defined at $\left( \ref{Definition of s hat}\right) $.
\end{notation}

The following Proposition proves that $F_{p}$ takes values in group-like
elements in $T\left( U\right) $ (denoted by $G\left( U\right) $). The result
helps to prove that the indefinite integral of a polynomial one-form along a
geometric rough path is again a geometric rough path.

\begin{proposition}
\label{Proposition F takes values in group}$F_{p}:G\left( V\right)
\rightarrow T\left( U\right) \ $is a lift of $f_{p}:V\rightarrow U$, and $%
F_{p}$ takes values in $G\left( U\right) $ the group-like elements in $%
T\left( U\right) $.
\end{proposition}

\begin{proof}
$F_{p}$ is a lift of $f_{p}$ because $f_{p}\left( s\right) =\widetilde{s}%
\left( \sigma _{1}\right) ,s\in G\left( V\right) $. For $\rho \in S_{j}$ and
$n_{1},\dots ,n_{j}\geq 1$, denote
\begin{equation*}
\rho \cdot m_{\succ }\left( 1_{n_{1}},\cdots ,1_{n_{j}}\right) :=m_{\succ
}\left( \overline{1_{n_{\rho \left( 1\right) }}},\cdots ,\overline{%
1_{n_{\rho \left( j\right) }}}\right)
\end{equation*}%
with $\overline{1_{n_{1}}}\left( i\right) :=i,\overline{1_{n_{l+1}}}\left(
i\right) :=\sum_{r=1}^{l}n_{r}+i$. Based on Lemma \ref{Proposition iterated
integrals of words} and Proposition \ref{Proposition free Lie group is a
group of characters}, for $s\in G\left( V\right) $, $\widehat{s}$ defined at
$\left( \ref{Definition of s hat}\right) $ satisfies%
\begin{eqnarray*}
&&conc\circ \left( \widehat{s}\left( m_{\succ }\left( 1_{n_{1}},\cdots
,1_{n_{k}}\right) \right) \otimes \widehat{s}\left( m_{\succ }\left(
1_{n_{k+1}},\cdots ,1_{n_{k+j}}\right) \right) \right) \\
&=&\widehat{s}\left( m_{\succ }\left( 1_{n_{1}},\cdots ,1_{n_{k}}\right)
\ast ^{\prime }m_{\succ }\left( 1_{n_{k+1}},\cdots ,1_{n_{k+j}}\right)
\right) \\
&=&\widehat{s}\left( \left( 1_{k}\ast ^{\prime }1_{j}\right) \cdot m_{\succ
}\left( \overline{1_{n_{1}}},\cdots ,\overline{1_{n_{k+j}}}\right) \right) .
\end{eqnarray*}%
Then based on the group structure of $\left\{ S_{k}\right\} _{k}$,%
\begin{equation*}
conc\circ \left( \widetilde{s}\left( \alpha \cdot \sigma _{k}\right) \otimes
\widetilde{s}\left( \rho \cdot \sigma _{j}\right) \right) =\widetilde{s}%
\left( \left( \alpha \ast ^{\prime }\rho \right) \cdot \sigma _{k+j}\right)
\end{equation*}%
for $\alpha \in S_{k},\rho \in S_{j},k\geq 1,j\geq 1$, where
\begin{equation*}
\widetilde{s}\left( \rho \cdot \sigma _{j}\right) :=\sum_{\substack{ %
k_{i}=0,\dots ,n  \\ i=1,\dots ,j}}\left( D^{k_{1}}p\right) \left( 0\right)
\otimes \cdots \otimes \left( D^{k_{j}}p\right) \left( 0\right) \widehat{s}%
\left( \rho \cdot m_{\succ }\left( 1_{n_{k+1}},\cdots ,1_{n_{k+j}}\right)
\right) .
\end{equation*}
\end{proof}

A polynomial one-form $p:V\rightarrow L\left( V,U\right) $ can be lifted to
an exact one-form $dF_{p}$ for $F_{p}:G\left( V\right) \rightarrow G\left(
U\right) $ defined at $\left( \ref{Definition of F}\right) $. Indeed, for $%
x\in BV\left( \left[ S,T\right] ,V\right) $, denote $X_{t}:=\exp \left(
x_{S}\right) S\left( x|_{\left[ S,t\right] }\right) $ for each $t$. Define $%
y\in BV\left( \left[ S,T\right] ,U\right) $ by $y_{t}:=\int_{r=S}^{t}p\left(
x_{r}\right) dx_{r}$, and define $Y_{t}:=S\left( y|_{\left[ S,t\right]
}\right) $. Then
\begin{equation}
Y_{S}^{-1}Y_{T}=F_{p}\left( X_{S}\right) ^{-1}F_{p}\left( X_{T}\right)
=\int_{r=S}^{T}dF_{p}dX_{r}.
\label{equality representation of almost multiplicative}
\end{equation}%
When $x_{S}=0$, $F_{p}\left( X_{S}\right) =1$ and the equality holds. When $%
x_{S}\neq 0$, by connecting $0$ and $x_{S}$ with a straight line and using
Chen's identity, the equality $Y_{S}^{-1}Y_{T}=F_{p}\left( X_{S}\right)
^{-1}\!F_{p}\left( X_{T}\right) $ still holds. The lift of $p$ to $dF_{p}$
does not depend on the path $x$.

For a general geometric rough path $X:\left[ S,T\right] \rightarrow G\left(
V\right) $, the integral of $p$ along $X$ can be defined as%
\begin{equation*}
\int_{r=S}^{T}p\left( x_{r}\right) dX_{r}:=F_{p}\left( X_{S}\right)
^{-1}F_{p}\left( X_{T}\right) =:\int_{r=S}^{T}dF_{p}dX_{r}\text{,}
\end{equation*}%
where $x$ denotes the projection of $X$ to a path in $V$. This integration
provides another interpretation of the almost multiplicative functional
defined by Lyons in Definition 3.2.2 \cite{lyons1998differential}.

The integration of polynomial one-forms provides basic ingredients for the
integration of general regular one-forms. Polynomials have nice
approximative properties, and classically the smoothness of a function is
expressed in terms of polynomials. Based on Stein (Chapter VI \cite%
{stein1970singular}), a function $\theta $ on a closed subset $F\subseteq
\mathbb{R}
^{n}$ is $\limfunc{Lip}\left( \gamma \right) $ for some $\gamma \in (k,k+1]$
if there exists a family of functions $\theta ^{j}$, $j=0,\dots ,k$, with $%
\theta ^{0}=\theta $, so that if%
\begin{equation*}
\theta ^{j}\left( x\right) =\sum_{\left\vert j+l\right\vert \leq k}\frac{%
\theta ^{j+l}\left( y\right) }{l!}\left( x-y\right) ^{l}+R_{j}\left(
x,y\right)
\end{equation*}%
then $\left\vert \theta ^{j}\left( x\right) \right\vert \leq M$ and $%
\left\vert R_{j}\left( x,y\right) \right\vert \leq M\left\vert
x-y\right\vert ^{\gamma -\left\vert j\right\vert }$ for all $x,y\in
F,\left\vert j\right\vert \leq k$. Based on Lyons \cite%
{lyons1998differential} Lipschitz one-forms are Lipschitz functions in the
sense of Stein, taking values in continuous linear mappings.

The following is Definition 3.2.2 and Theorem 3.2.1 p.285 \cite%
{lyons1998differential}.

\begin{definition}[Lyons]
\label{Definition almost multiplicative function Y}For any multiplicative
functional $X_{s,t}$ in $\Omega G\left( V\right) ^{p}$ define%
\begin{equation*}
Y_{s,t}^{i}=\sum_{l_{1},\dots ,l_{i}=1}^{\left[ p\right] }\theta
^{l_{1}}\left( x_{s}\right) \otimes \cdots \otimes \theta ^{l_{i}}\left(
x_{s}\right) \sum_{\pi \in \Pi _{\underline{l}}}\pi \left(
X_{s,t}^{\left\Vert \underline{l}\right\Vert }\right)
\end{equation*}
\end{definition}

\begin{theorem}[Lyons, Existence of Integral]
\label{Theorem existence of the integration Lyons}For any multiplicative
functional $X_{s,t}$ in $\Omega G\left( V\right) ^{p}$ and any one-form $%
\theta \in \limfunc{Lip}[\gamma -1,\{X_{u},u\in \left[ s,t\right] \}]$ with $%
\gamma >p$ the sequence $Y_{s,t}=(1,Y_{s,t}^{1},\dots ,Y_{s,t}^{[p]})$
defined above is almost multiplicative and of finite $p$-variation; if $%
X_{s,t}$ is controlled by $\omega $ on $J$ where $\omega $ is bounded by $L$%
, and the $\limfunc{Lip}[\gamma -1]$ norm of $\theta $ is bounded by $M$,
then the almost multiplicative and $p$-variation properties of $Y$ are
controlled by multiples of $\omega $ which depend only on $\gamma ,p,L,M$.
\end{theorem}

The multiplicative functional associated with $Y$ obtained in Theorem \ref%
{Theorem existence of the integration Lyons} is defined to be the integral
of the one-form $\theta $ along geometric rough path $X$ Theorem 3.3.1 p.274
Definition 3.2.3 p.288 \cite{lyons1998differential}. Denote the integral as $%
\int \theta \left( x\right) dX$. Based on Theorem 3.3.1 p.274 \cite%
{lyons1998differential},%
\begin{equation}
\int_{r=0}^{1}\theta \left( x_{r}\right) dX_{r}:=\lim_{\left\vert
D\right\vert \rightarrow 0,D\subset \left[ 0,1\right] }Y_{t_{0},t_{1}}\cdots
Y_{t_{n-1},t_{n}}\text{.}  \label{definition of integral}
\end{equation}

Based on the lift of polynomial one-form $p$ to exact one-form $dF_{p}$, the
integral $\int \theta \left( x\right) dX$ can be interpreted in terms of
time-varying exact one-forms.

Let $X:\left[ 0,1\right] \rightarrow G\left( V\right) $ be a geometric $p$%
-rough path, and let $\theta :V\rightarrow L\left( V,U\right) $ be a $%
\limfunc{Lip}\left( \gamma \right) $ one-form for $\gamma >p-1$. Let $x$
denote the projection of $X$ to a path in $V$. For $x_{s}\in V$, define the
polynomial one-form $p_{x_{s}}:V\rightarrow L\left( V,U\right) $ as%
\begin{equation}
p_{x_{s}}\left( v\right) \left( w\right) =\sum_{k=0}^{\left[ p\right]
}\theta ^{k}\left( x_{s}\right) \frac{\left( v-x_{s}\right) ^{\otimes k}}{k!}%
\left( w\right)  \label{polynomial one-form pv0}
\end{equation}%
for $v,w\in V$. For the polynomial one-form $p_{x_{s}}$, define $%
F_{p_{x_{s}}}$ as at $\left( \ref{Definition of F}\right) $. The almost
multiplicative functional $Y_{s,t}\ $in Definition \ref{Definition almost
multiplicative function Y} can be expressed as%
\begin{equation}
Y_{s,t}=F_{p_{x_{s}}}\left( X_{s}\right) ^{-1}F_{p_{x_{s}}}\left(
X_{t}\right) =:\int_{r=s}^{t}dF_{p_{x_{s}}}dX_{r}  \label{expression of Yst}
\end{equation}%
for every $s\leq t$. The equality $\left( \ref{expression of Yst}\right) $
follows from a generalized Chen's identity about the multiplicativity of
rough path liftings of controlled paths/effects (see Section \ref{Section
dominated paths} for details). The generalized Chen's identity can be proved
based on the uniqueness of the continuous lifting (Proposition \ref%
{Proposition stability under iterated integration}).

\begin{theorem}
\label{Theorem integration interpretation}For a geometric $p$-rough path $X:%
\left[ 0,1\right] \rightarrow G\left( V\right) $ and a $\limfunc{Lip}\left(
\gamma -1\right) $ one-form $\theta :V\rightarrow L\left( V,U\right) $ for $%
\gamma >p$, let $\int_{r=0}^{1}\theta \left( x_{r}\right) dX_{r}$ denote the
integral defined by Lyons in \cite{lyons1998differential}. Then with $%
p_{x_{s}}\ $defined at $\left( \ref{polynomial one-form pv0}\right) $ and $%
F_{p_{x_{s}}}\ $defined at $\left( \ref{Definition of F}\right) $,%
\begin{eqnarray*}
&&\int_{r=0}^{1}\theta \left( x_{r}\right) dX_{r} \\
&=&\lim_{\left\vert D\right\vert \rightarrow 0,D\subset \left[ 0,1\right]
}\int_{r=t_{0}}^{t_{1}}dF_{p_{x_{t_{0}}}}dX_{r}\cdots
\int_{r=t_{n-1}}^{t_{n}}dF_{p_{x_{t_{n-1}}}}dX_{r} \\
&=&:\int_{r=0}^{1}dF_{p_{x_{r}}}dX_{r}
\end{eqnarray*}%
where $D=\left\{ t_{k}\right\} _{k=0}^{n},0=t_{0}<\cdots <t_{n}=1,n\geq 1$
with $\left\vert D\right\vert :=\max_{k}\left\vert t_{k+1}-t_{k}\right\vert $%
.
\end{theorem}

The equality is based on $\left( \ref{definition of integral}\right) $ and $%
\left( \ref{expression of Yst}\right) $; the existence of the integral $\int
\theta \left( x\right) dX$ is obtained in Theorem 3.2.1 \cite%
{lyons1998differential} i.e. Theorem \ref{Theorem existence of the
integration Lyons}.

There is a minor difference between geometric rough paths $\Omega G\left(
V\right) ^{p}$ in \cite{lyons1998differential} and that defined in
Definition \ref{Definition geometric rough paths} (paths in Definition \ref%
{Definition geometric rough paths} are also called weak geometric rough
paths). The integration $\int_{r=0}^{1}dF_{p_{x_{r}}}dX_{r}$ in Theorem \ref%
{Theorem integration interpretation} can be applied to both classes of rough
paths.

The integration of time-varying exact one-forms exists in a general setting.
Consider two groups $G_{1}$ and $G_{2}$, and a path $X:\left[ 0,1\right]
\rightarrow G_{1}$. Suppose $f_{t}:G_{1}\rightarrow G_{2}\ $is a family of
functions indexed by $t\in \left[ 0,1\right] $. If the limit exists in $G_{2}
$:%
\begin{equation*}
\lim_{\left\vert D\right\vert \rightarrow 0,D=\left\{ t_{k}\right\}
_{k=0}^{n}\subset \left[ 0,1\right] }\int_{r=t_{0}}^{t_{1}}df_{t_{0}}dX_{r}%
\int_{r=t_{1}}^{t_{2}}df_{t_{1}}dX_{r}\cdots
\int_{r=t_{n-1}}^{t_{n}}df_{t_{n-1}}dX_{r}
\end{equation*}%
where $\int_{r=t_{k}}^{t_{k+1}}df_{t_{k}}dX_{r}:=f_{t_{k}}\left(
X_{t_{k}}\right) ^{-1}f_{t_{k}}\left( X_{t_{k+1}}\right) $, then the\
integral $\int_{r=0}^{1}df_{r}dX_{r}$ is defined to be the limit. A
sufficient condition for the existence of the integral is given in Theorem %
\ref{Theorem Integration} below based on \textquotedblleft a non-commutative
sewing lemma\textquotedblright\ by Feyel, de la Pradelle and Mokobodzki \cite%
{feyel2008non}. The integral can be viewed as an non-abelian analogue of
Young's integral \cite{young1936inequality}.

The integration of time-varying exact one-forms can also be explained via
\textit{reset} of functions. For $f:G_{1}\rightarrow G_{2}$, define the
reset of $f$ at $a\in G_{1}$ as%
\begin{eqnarray*}
f_{a}:G_{1} &\rightarrow &G_{2} \\
g &\mapsto &f\left( a\right) ^{-1}f\left( ag\right) .
\end{eqnarray*}%
The reset of functions are consistent with the integration of exact
one-forms. For $X:\left[ 0,1\right] \rightarrow G_{1}$, one has $%
\int_{r=0}^{1}dfdX_{r}=f\left( X_{0}\right) ^{-1}f\left( X_{1}\right)
=f_{X_{0}}\left( X_{0}^{-1}X_{1}\right) $. Consider a principal bundle $P$
on $G_{1}$ that associates each $a\in G_{1}$ with the group of functions $%
\left\{ h|h:G_{1}\rightarrow G_{2},h\left( 1_{G_{1}}\right)
=1_{G_{2}}\right\} $. The reset of functions defines a parallel
transportation on $P$. For $h\in P_{a},a\in G_{1}$, the parallel translation
of $h$ is given by
\begin{equation*}
h_{b}:=\left( h_{a}\right) _{a^{-1}b}
\end{equation*}%
for $b\in G_{1}$. The $\left\{ h_{b}|b\in G_{1}\right\} $ defined are
consistent:%
\begin{equation*}
\left( \left( h_{a}\right) _{a^{-1}b}\right) _{b^{-1}c}=\left( h_{a}\right)
_{a^{-1}c}
\end{equation*}%
for $b,c\in G_{1}$. For $X:\left[ 0,1\right] \rightarrow G_{1}$, consider
\begin{equation*}
\beta \in \left( X,P_{X}\right) \text{ i.e. }\beta :X_{t}\mapsto \beta
\left( X_{t}\right) \in P_{X_{t}}\text{.}
\end{equation*}%
Then the integration of time-varying exact one-forms (when exists)
\begin{equation*}
\int_{r=0}^{1}df_{r}dX_{r}:=\lim_{\left\vert D\right\vert \rightarrow
0,D=\left\{ t_{k}\right\} \subset \left[ 0,1\right] }%
\int_{r=t_{0}}^{t_{1}}df_{t_{0}}dX_{r}\cdots
\int_{r=t_{n-1}}^{t_{n}}df_{t_{n-1}}dX_{r}
\end{equation*}%
can be equivalently defined as%
\begin{equation*}
\int_{r=0}^{1}\beta \left( X_{r}\right) dX_{r}:=\lim_{\left\vert
D\right\vert \rightarrow 0,D=\left\{ t_{k}\right\} \subset \left[ 0,1\right]
}\beta \left( X_{t_{0}}\right) \left( X_{t_{0},t_{1}}\right) \cdots \beta
\left( X_{t_{n-1}}\right) \left( X_{t_{n-1},t_{n}}\right)
\end{equation*}%
where $\beta \left( X_{t}\right) \left( g\right) :=f\left( X_{t}\right)
^{-1}f\left( X_{t}g\right) =f_{X_{t}}\left( g\right) ,g\in G_{1},t\in \left[
0,1\right] $, and $\left\{ \beta \left( X_{t}\right) \right\} _{t}$ are
compared after parallel translation.

For a fixed continuous path $X:\left[ 0,1\right] \rightarrow G_{1}$,
consider a condition on exact one-forms $\left( df_{t}\right) _{t}$ for $%
f_{t}:G_{1}\rightarrow G_{2}$ that guarantees the existence of the integral $%
\int_{r=0}^{1}df_{r}dX_{r}$. Theorem \ref{Theorem Integration} below gives a
condition that roughly states that, if one-step discrete approximations are
comparable to two-steps discrete approximations up to a small error, then
the integral exists as a limit of Riemann products. When $X$ is a geometric $%
p$-rough path, the condition on $\left( df_{t}\right) _{t}$ can be further
specified, and the condition can be viewed as an inhomogeneous analogue of
Young's condition p.264 \cite{young1936inequality}. Such a condition is
closely related to the notion of weakly controlled paths introduced by
Gubinelli \cite{gubinelli2004controlling}. The following is Definition 1
\cite{gubinelli2004controlling}.

\begin{definition}[Gubinelli, weakly controlled paths]
Fix an interval $I\subseteq
\mathbb{R}
$ and let $X\in \mathcal{C}^{\gamma }\left( I,V\right) $. A path $Z\in
\mathcal{C}^{\gamma }\left( I,V\right) $ is said to be weakly controlled by $%
X$ in $I$ with a reminder of order $\eta $ if there exists a path $Z^{\prime
}\in \mathcal{C}^{\eta -\gamma }\left( I,V\otimes V^{\ast }\right) $ and a
process $R_{Z}\in \Omega C^{\eta }\left( I,V\right) $ with $\eta >\gamma $
such that%
\begin{equation*}
\delta Z^{\mu }=Z^{\prime \mu \nu }\delta X^{\nu }+R_{Z}^{\mu }\text{.}
\end{equation*}%
If this is the case we will write $\left( Z,Z^{\prime }\right) \in \mathcal{D%
}_{X}^{\gamma ,\eta }\left( I,V\right) $ and we will consider on the linear
space $\mathcal{D}_{X}^{\gamma ,\eta }\left( I,V\right) $ the semi-norm%
\begin{equation*}
\left\Vert Z\right\Vert _{\mathcal{D}\left( X,\gamma ,\eta \right)
,I}:=\left\Vert Z^{\prime }\right\Vert _{\infty ,I}+\left\Vert Z^{\prime
}\right\Vert _{\eta -\gamma ,I}+\left\Vert R_{Z}\right\Vert _{\eta
,I}+\left\Vert Z\right\Vert _{\gamma ,I}\text{.}
\end{equation*}
\end{definition}

For a fixed geometric rough path $X$, we consider a class of integrable
one-forms whose indefinite integrals are called effects (see Section \ref%
{Section dominated paths} for more details). The set of effects of $X$ is a
subset of the paths controlled by $X$. The relationship between controlled
paths and effects is comparable to that between the integrand and the
integral. In the integration of one-form $\int \alpha \left( x\right) dX$, $%
t\mapsto \alpha \left( x_{t}\right) $ is a controlled path; $t\mapsto
\int_{r=0}^{t}\alpha \left( x_{r}\right) dX_{r}$ is an effect so a
controlled path. A controlled path can also be interpreted as a time-varying
exact one-form, and its varying speed can be a little quicker than that of
an effect. One benefit of working with effects is that basic operations
(multiplication, composition with regular functions, integration, iterated
integration) are continuous operations in the space of one-forms in operator
norm. In particular, the lifting of an effect to a geometric rough path is
continuous. In \cite{lyons2015theory} effects (dominated paths) are employed
to give a short proof of the unique solvability and stability of the
solution to differential equations driven by rough paths, and differences
between adjacent Picard iterations decay factorially in operator norm.

\section{Integration of time-varying exact one-forms}

\label{Section generalized integration}

For continuous $x:\left[ 0,1\right] \rightarrow
\mathbb{C}
$ of finite $p$-variation and continuous $y:\left[ 0,1\right] \rightarrow
\mathbb{C}
$ of finite $q$-variation $p^{-1}+q^{-1}>1,p\geq 1,q\geq 1$, Young \cite%
{young1936inequality} defined the Stieltjes integral $%
\int_{r=0}^{1}x_{r}dy_{r}$ as the limit of Riemann sums. Lyons \cite%
{lyons1998differential} defined the integration of one-forms along geometric
rough paths by constructing a multiplicative functional from an almost
multiplicative functional.

In \cite{feyel2008non} Feyel, La Pradelle and Mokobodzki proved a
non-commutative sewing lemma that constructs multiplicative functions taking
values in a monoid with a distance.

Based on \cite{feyel2008non}, let $M$ be a monoid with a unit element $I$,
and $M$ is complete under a distance $d$ that satisfies%
\begin{equation}
d\left( xz,yz\right) \leq \left\vert z\right\vert d\left( x,y\right) ,\text{%
\ }d\left( zx,zy\right) \leq \left\vert z\right\vert d\left( x,y\right)
\label{metric on monoid}
\end{equation}%
for $x,y,z\in G_{2}$, where $z\mapsto \left\vert z\right\vert $ is a
Lipschitz function on $M$ with $\left\vert I\right\vert =1$.

Suppose $V:[0,T)\rightarrow \overline{%
\mathbb{R}
^{+}}$ is a \textit{strong control function} i.e. $V\left( 0\right) =0$,
non-decreasing, and there exsits a $\theta >2$ such that for every $t$%
\begin{equation*}
\overline{V}\left( t\right) :=\sum_{n\geq 0}\theta ^{n}V\left(
t2^{-n}\right) <\infty \text{.}
\end{equation*}%
For example, $V\left( t\right) =t^{\alpha }$ when $\alpha >1$ is a strong
control function.

Suppose $\mu :\left\{ \left( s,t\right) |0\leq s\leq t<T\right\} \rightarrow
\left( M,d\right) $ is continuous, $\mu \left( t,t\right) =I$ for every $t$,
and
\begin{equation}
d\left( \mu \left( s,t\right) ,\mu \left( s,u\right) \mu \left( u,t\right)
\right) \leq V\left( t-s\right)  \label{condition on mu}
\end{equation}%
for every $s\leq u\leq t$. $u:\left\{ \left( s,t\right) |0\leq s\leq
t<T\right\} \rightarrow \left( M,d\right) $ is called \textit{multiplicative}
if $u\left( s,t\right) =u\left( s,u\right) u\left( u,t\right) $ for every $%
s\leq u\leq t$.

\begin{theorem}[Feyel, La Pradelle, Mokobodzki]
\label{Theorem Feyel, LaPradelle and Mokobodzki}There exists a unique
continuous multiplicative function $u$ such that $d\left( \mu \left(
s,t\right) ,u\left( s,t\right) \right) \leq \limfunc{Cst}\overline{V}\left(
t-s\right) $ for every $s\leq t$.
\end{theorem}

Let $G_{1}$ and $G_{2}$ be two groups, and suppose $G_{2}$ is complete under
a distance $d$ that satisfies $\left( \ref{metric on monoid}\right) $. Let $%
X:\left[ 0,T\right] \rightarrow G_{1}$. Suppose $f_{t}:G_{1}\rightarrow
G_{2} $ is a family of functions indexed by $t\in \left[ 0,T\right] $.
Define $\mu :\left\{ \left( s,t\right) |0\leq s\leq t<T\right\} \rightarrow
\left( G_{2},d\right) $ as%
\begin{equation}
\mu \left( s,t\right) :=f_{s}\left( X_{s}\right) ^{-1}\!f_{s}\left(
X_{t}\right) \in G_{2}  \label{definition of mu}
\end{equation}
for $s\leq t$. Suppose $\left( \ref{condition on mu}\right) $ holds for a
strong control function $V$.

Let $u:\left\{ \left( s,t\right) |0\leq s\leq t<T\right\} \rightarrow \left(
G_{2},d\right) \ $denote the multiplicative function associated with $\mu $
at $\left( \ref{definition of mu}\right) $ obtained by Theorem \ref{Theorem
Feyel, LaPradelle and Mokobodzki}.

\begin{theorem}
\label{Theorem Integration}Define $\int_{r=0}^{\cdot }df_{r}dX_{r}:\left[ 0,T%
\right] \rightarrow G_{2}$ as $\int_{r=0}^{t}df_{r}dX_{r}:=u\left(
0,t\right) $ for every $t$. Then $\int_{r=0}^{\cdot }df_{r}dX_{r}$ is the
unique continuous path $y:\left[ 0,T\right] \rightarrow G_{2}$ such that $%
y_{0}=1_{G_{2}}$ and $d(y_{s}^{-1}y_{t},\int_{r=s}^{t}d\!f_{s}dX_{r})\leq
\limfunc{Cst}\overline{V}\left( t-s\right) $ for every $s\leq t$.
\end{theorem}

Based on the definition, for $X:\left[ 0,1\right] \rightarrow G_{1}$ and $%
f:G_{1}\rightarrow G_{2}$, $\int_{r=0}^{1}dfdX_{r}:=f\left( X_{0}\right)
^{-1}f\left( X_{1}\right) $.

Based p.31 \cite{feyel2008non} Theorem \ref{Theorem Integration} applies
when $G_{2}$ is a group of elements beginning with $1$ in the\ tensor
algebra over a Banach space when tensor powers are equipped with admissible
norms.

\section{Effects of a geometric rough path}

\label{Section dominated paths}

The set of effects of a geometric rough path is a subset of the paths
controlled by the geometric rough path. Similar to controlled paths, effects
are stable under basic operations. Integrals of one-forms and solution to
differential equations are effects so are controlled paths.

Let $V$ and $U$ be two Banach spaces, and let $G\left( V\right) $ be the set
of group-like elements in the tensor algebra $T\left( V\right) $. For $p\geq
1$, set $\left[ p\right] :=\max \left\{ n|n\in
\mathbb{N}
,n\leq p\right\} $. For $k=1,\dots ,\left[ p\right] $, denote by $L\left(
V^{\otimes k},U\right) $ the set of continuous linear operators from $%
V^{\otimes k}$ to $U$.

\begin{notation}
Let $E^{U}$ denote the vector bundle on $G\left( V\right) $ that associates
each $a\in G\left( V\right) $ with the vector space:%
\begin{equation}
E_{a}^{U}:=\left\{ \phi \bigg|\phi :G\left( V\right) \rightarrow U,\phi
=\sum_{k=1}^{\left[ p\right] }\phi ^{k},\phi ^{k}\in L\left( V^{\otimes
k},U\right) \right\}  \label{definition of vector bundle E}
\end{equation}%
where $\phi ^{k}\left( x\right) :=\phi ^{k}x^{k}$ for $x\in G\left( V\right)
,x=\sum_{k\geq 0}x^{k},x^{k}\in V^{\otimes k}$.
\end{notation}

$E_{a}^{U}$ can be considered as the space of `polynomials' up to degree-$%
\left[ p\right] $ that has no `constant', and can be viewed as the
polynomial approximation to the `tangent space' of functions $\left\{
f_{a}|f:G\left( V\right) \rightarrow U\right\} $ with $f_{a}\left( x\right)
:=f\left( ax\right) -f\left( a\right) ,x\in G\left( V\right) $.

The parallel transportation on $E^{U}$ is given by the reset of functions.

\begin{notation}
For $p_{a}\in E_{a}^{U}$ and $b\in G\left( V\right) $, define $\left(
p_{a}\right) _{a^{-1}b}\in E_{b}^{U}$ as%
\begin{equation*}
\left( p_{a}\right) _{a^{-1}b}\left( x\right) :=p_{a}\left( a^{-1}bx\right)
-p_{a}\left( a^{-1}b\right)
\end{equation*}%
for $x\in G\left( V\right) $.
\end{notation}

Then $\left( \left( p_{a}\right) _{a^{-1}b}\right) _{b^{-1}c}=\left(
p_{a}\right) _{a^{-1}c}$ for $a,b,c\in G\left( V\right) $.

For $\phi \in E_{a},\phi =\sum_{k=1}^{\left[ p\right] }\phi ^{k}$, denote
\begin{equation*}
\left\Vert \phi \right\Vert :=\max_{k=1,\dots ,\left[ p\right] }\left\Vert
\phi ^{k}\right\Vert \text{ and }\left\Vert \phi \right\Vert
_{k}:=\left\Vert \phi ^{k}\right\Vert
\end{equation*}%
where $\left\Vert \phi ^{k}\right\Vert $ denotes the norm of $\phi ^{k}$ as
a linear operator.

\begin{definition}[Operator Norm]
\label{Definition operator norm}Let $X:\left[ 0,1\right] \rightarrow G\left(
V\right) $ be a geometric $p$-rough path for some $p\geq 1$, and let $U$ be
a Banach space. Suppose
\begin{equation*}
\beta \in \left( X,E_{X}^{U}\right) \text{ i.e. }\beta :X_{t}\mapsto \beta
\left( X_{t}\right) \in E_{X_{t}}^{U}.
\end{equation*}%
For $t\in \left[ 0,1\right] $ and $a\in G\left( V\right) $, define%
\begin{eqnarray*}
\left( \beta \left( X_{t}\right) \right) _{a} &\in &E_{X_{t}a}^{U} \\
x &\mapsto &\beta \left( X_{t}\right) \left( ax\right) -\beta \left(
X_{t}\right) \left( a\right) ,x\in G\left( V\right) .
\end{eqnarray*}%
For a control $\omega $ and $\theta >1$, define the operator norm
\begin{equation*}
\left\Vert \beta \right\Vert _{\theta }^{\omega }:=\sup_{t\in \left[ 0,1%
\right] }\left\Vert \beta \left( X_{t}\right) \right\Vert +\max_{k=1,\dots ,%
\left[ p\right] }\sup_{0\leq s<t\leq 1}\frac{\left\Vert \beta \left(
X_{t}\right) -\left( \beta \left( X_{s}\right) \right)
_{X_{s}^{-1}X_{t}}\right\Vert _{k}}{\omega \left( s,t\right) ^{\theta -\frac{%
k}{p}}}\text{.}
\end{equation*}
\end{definition}

\begin{definition}[Slowly-Varying One-Form]
\label{Definition slowly-varying one-form}Let $X:\left[ 0,1\right]
\rightarrow G\left( V\right) $ be a geometric $p$-rough path for some $p\geq
1$, and let $U$ be a Banach space. Then $\beta \in \left( X,E_{X}^{U}\right)
$ is called a slowly varying one-form, if there exists a control $\omega $
and $\theta >1$ such that $\left\Vert \beta \right\Vert _{\theta }^{\omega
}<\infty $.
\end{definition}

For each $t\in \left[ 0,1\right] $, $\beta \left( X_{t}\right) $ can be
viewed as a continuous linear mapping from monomials (components of rough
paths) to the vector space $U$.

For a control $\omega \,\ $and $\theta >1$, the set of slowly varying
one-forms along $X$ with finite operator norm $\left\Vert \cdot \right\Vert
_{\theta }^{\omega }$ form a Banach space.

Suppose $\beta \in \left( X,E_{X}^{U}\right) $ is a slowly-varying one-form.
Let $\beta \left( X_{t}\right) _{X_{t}^{-1}}\ $be considered as functions $%
G\left( V\right) \rightarrow U$ index by $t\in \left[ 0,1\right] $. Define%
\begin{eqnarray*}
&&\int_{r=0}^{1}\beta \left( X_{r}\right) dX_{r} \\
&:&=\int_{r=0}^{1}d\left( \beta \left( X_{r}\right) _{X_{r}^{-1}}\right)
dX_{r} \\
&:&=\lim_{\left\vert D\right\vert \rightarrow 0,D\subset \left[ 0,1\right]
}\left( \beta \left( X_{t_{k}}\right) _{X_{t_{k}}^{-1}}\left(
X_{t_{k+1}}\right) -\beta \left( X_{t_{k}}\right) _{X_{t_{k}}^{-1}}\left(
X_{t_{k}}\right) \right) \\
&=&\lim_{\left\vert D\right\vert \rightarrow 0,D\subset \left[ 0,1\right]
}\sum_{k,t_{k}\in D}\beta \left( X_{t_{k}}\right) \left(
X_{t_{k},t_{k+1}}\right)
\end{eqnarray*}%
where the last equality is based on $\beta \left( X_{t_{k}}\right) \left(
1_{G\left( V\right) }\right) =0$. The integral exists based on the slowly
varying condition on $\beta $.

\begin{definition}[Effects]
Let $X:\left[ 0,1\right] \rightarrow G\left( V\right) $ be a geometric $p$%
-rough path for some $p\geq 1$, and let $\beta \in \left( X,E_{X}^{U}\right)
$ be a slowly varying one-form. Then for $\xi \in U$, the integral path
\begin{equation*}
t\mapsto \xi +\int_{r=0}^{t}\beta \left( X_{r}\right) dX_{r},t\in \left[ 0,1%
\right]
\end{equation*}%
is called an effect of $X$.
\end{definition}

\begin{theorem}
\label{Proposition continuity of integral w.r.t. one-form}Suppose $\beta \in
\left( X,E_{X}^{U}\right) $ is a slowly-varying one-form such that $%
\left\Vert \beta \right\Vert _{\theta }^{\omega }<\infty $ for a control $%
\omega $ and $\theta >1$. Define $h:\left[ 0,1\right] \rightarrow U$ as
\begin{equation*}
h_{t}:=\int_{r=0}^{t}\beta \left( X_{r}\right) dX_{r},t\in \left[ 0,1\right]
.
\end{equation*}%
Then with control $\hat{\omega}:=\omega +\left\Vert X\right\Vert
_{p-var}^{p} $,%
\begin{equation*}
\left\Vert h_{t}-h_{s}-\beta _{s}\left( X_{s}\right) \left( X_{s,t}\right)
\right\Vert \leq C_{p,\theta ,\hat{\omega}\left( 0,T\right) }\left\Vert
\beta \right\Vert _{\theta }^{\omega }\hat{\omega}\left( s,t\right) ^{\theta
}
\end{equation*}%
for every $s\leq t$, and%
\begin{equation*}
\left\Vert h\right\Vert _{p-var,\left[ 0,T\right] }\leq C_{p,\theta ,\hat{%
\omega}\left( 0,T\right) }\left\Vert \beta \right\Vert _{\theta }^{\omega }%
\text{.}
\end{equation*}
\end{theorem}

The first estimate can be proved similarly to result 5 on p.254 Young \cite%
{young1936inequality}; the second estimate follows from the first.

\section{Stability of Effects under Basic Operations}

Consider the set of effects of a geometric rough path. Effects are closed
under basic operations (multiplication, composition with regular functions,
integration, iterated integration); the proof is similar to that for
controlled paths as in Proposition 4 on p.100 and Theorem 1 on p.101
Gubinelli \cite{gubinelli2004controlling}. For effects these operations are
continuous in the space of one-forms in operator norm.

In this section, we fix a geometric $p$-rough path $X:\left[ 0,1\right]
\rightarrow G\left( V\right) $ for some $p\geq 1$, and consider the set of
effects of $X$.

For $k\geq 1$ let $1_{k}$ denote the identity element in $S_{k}$. Let $\cdot
:%
\mathbb{Z}
S\times T\left( V\right) \rightarrow T\left( V\right) $ denote the bilinear
map given by $\sigma \cdot \left( v_{1}\otimes \cdots \otimes v_{n}\right)
:=v_{\sigma 1}\otimes \cdots \otimes v_{\sigma n}$ for $\sigma \in
S_{n},v_{i}\in V$. For $p\geq 1,\gamma \geq 1$, denote $\left[ p\right]
:=\max \left\{ n|n\in
\mathbb{N}
,n\leq p\right\} $ and $\lfloor \gamma \rfloor :=\max \left\{ n|n\in
\mathbb{N}
,n<\gamma \right\} $.

\subsection{Composition with Regular Functions}

The stability of effects under composition with regular functions follows
from the fact that polynomials are closed under composition.

For Banach spaces $U$ and $W$, denote by $C^{\gamma }\left( U,W\right) $ the
set of functions $\varphi :U\rightarrow W$ that are $\lfloor \gamma \rfloor $%
-times Fr\'{e}chet differentiable with the $\lfloor \gamma \rfloor $th
derivative $\left( \gamma -\lfloor \gamma \rfloor \right) $-H\"{o}lder,
uniformly on any bounded set. That is, for each $R>0$,%
\begin{equation*}
\left\Vert \varphi \right\Vert _{\gamma ,R}:=\max_{k=0,1,\dots ,\lfloor
\gamma \rfloor }\left\Vert \left( D^{k}\varphi \right) \right\Vert _{\infty
,R}+\left\Vert \left( D^{\lfloor \gamma \rfloor }\varphi \right) \right\Vert
_{\left( \gamma -\lfloor \gamma \rfloor \right) -\text{H\"{o}l},R}
\end{equation*}%
where $\left\Vert \cdot \right\Vert _{\infty ,R}$ resp. $\left\Vert \cdot
\right\Vert _{\left( \gamma -\lfloor \gamma \rfloor \right) -\text{H\"{o}l}%
,R}$ denote the uniform resp. H\"{o}lder norm on $\left\{ u\in U|\left\Vert
u\right\Vert \leq R\right\} $.

For $\phi \left( X_{t}\right) \in E_{X_{t}}^{U}$ and $l=1,\dots ,\left[ p%
\right] $, define the `truncated polynomial':%
\begin{multline*}
\prod\nolimits_{\leq \left[ p\right] }\left( \phi \left( X_{t}\right)
^{\otimes l}\right) \in E_{X_{t}}^{U^{\otimes l}} \\
x\mapsto \sum_{\substack{ k_{1}+\cdots +k_{l}\leq \left[ p\right]  \\ %
k_{i}=1,\dots ,\left[ p\right] }}\left( \phi ^{k_{1}}\left( X_{t}\right)
\otimes \cdots \otimes \phi ^{k_{l}}\left( X_{t}\right) \right) \left(
\left( 1_{k_{1}}\ast ^{\prime }\cdots \ast ^{\prime }1_{k_{l}}\right) \cdot
\left( x^{k_{1}+\cdots +k_{l}}\right) \right)
\end{multline*}%
for $x\in G\left( V\right) $, where $\phi =\sum_{k=1}^{\left[ p\right] }\phi
^{k},\phi ^{k}\in L\left( V^{\otimes k},U\right) $ and $x=\sum_{k\geq
0}x^{k},x^{k}\in V^{\otimes k}$.

\begin{proposition}
Suppose $\beta _{1}\in \left( X,E_{X}^{U}\right) $ is a slowly-varying
one-form such that $\left\Vert \beta _{1}\right\Vert _{\omega _{1}}^{\theta
_{1}}<\infty $ for a control $\omega _{1}$ and $\theta _{1}>1$. Denote $%
h_{t}:=\int_{0}^{t}\beta _{1}\left( X_{t}\right) dX_{t},t\in \left[ 0,1%
\right] $. For $\varphi \in C^{\gamma }\left( U,W\right) $, $\gamma >p$,
define $\beta \in \left( X,E_{X}^{W}\right) \,$as%
\begin{equation*}
\beta \left( X_{t}\right) \left( x\right) :=\sum_{l=1}^{\left[ p\right] }%
\frac{1}{l!}\left( D^{l}\varphi \right) \left( h_{t}\right) \left(
\prod\nolimits_{\leq \left[ p\right] }\left( \beta _{1}\left( X_{t}\right)
^{\otimes l}\right) \right) \left( x\right)
\end{equation*}%
for $x\in G\left( V\right) $. Then with $\omega :=\omega _{1}+\left\Vert
X\right\Vert _{p-var}^{p}$ and $\theta :=\min \left( \theta _{1},\frac{%
\gamma }{p},\frac{\left[ p\right] +1}{p}\right) $,%
\begin{equation}
\left\Vert \beta \right\Vert _{\theta }^{\omega }\leq C_{p,\theta ,\omega
\left( 0,1\right) }\left\Vert \varphi \right\Vert _{\gamma ,\left\Vert
h\right\Vert _{\infty }}\max \left( \left\Vert \beta _{1}\right\Vert
_{\theta _{1}}^{\omega _{1}},\left( \left\Vert \beta _{1}\right\Vert
_{\theta _{1}}^{\omega _{1}}\right) ^{\left[ p\right] }\right) \text{,}
\label{Composition operator norm}
\end{equation}%
where $\left\Vert h\right\Vert _{\infty }:=\sup_{t\in \left[ 0,1\right]
}\left\Vert h_{t}\right\Vert $, and
\begin{equation*}
\int_{r=0}^{t}\beta \left( X_{r}\right) dX_{r}=\varphi \left( h_{t}\right)
-\varphi \left( h_{0}\right) \text{ for }t\in \left[ 0,1\right] \text{.}
\end{equation*}
\end{proposition}

\begin{remark}
Effects are closed under multiplication and form an algebra. For Banach
spaces $U_{i},i=1,2$, consider $\varphi :\left( U_{1},U_{2}\right)
\rightarrow U_{1}\otimes U_{2}$ given by $\left( u_{1},u_{2}\right) \mapsto
u_{1}\otimes u_{2}$. Then $\varphi $ is smooth and $D^{3}\varphi \equiv 0$.
\end{remark}

\begin{proof}
For $l=1,\dots ,\left[ p\right] $,%
\begin{equation*}
\left( \phi ^{\otimes l}\right) _{a}\left( x\right) =\left( \phi \left(
a\right) +\phi _{a}\left( x\right) \right) ^{\otimes l}-\phi ^{\otimes
l}\left( a\right)
\end{equation*}%
for $x\in G\left( V\right) $.

We rescale $\varphi $ by $\left\Vert \varphi \right\Vert _{\gamma
,\left\Vert h\right\Vert _{\infty }}^{-1}$ and assume $\left\Vert \varphi
\right\Vert _{\gamma ,\left\Vert h\right\Vert _{\infty }}=1$. Denote $%
X_{s,t}:=X_{s}^{-1}X_{t}$. For $s\leq t$,%
\begin{eqnarray*}
&&\left( \beta \left( X_{t}\right) -\left( \beta \left( X_{s}\right) \right)
_{X_{s,t}}\right) \left( x\right) \\
&=&\sum_{l=1}^{\left[ p\right] }\frac{1}{l!}\left( \left( D^{l}\varphi
\right) \left( h_{t}\right) -\sum_{j=0}^{\left[ p\right] -l}\frac{1}{j!}%
\left( D^{j+l}\varphi \right) \left( h_{s}\right) \left( h_{t}-h_{s}\right)
^{\otimes l}\right) \left( \prod\nolimits_{\leq \left[ p\right] }\left(
\beta _{1}\left( X_{t}\right) ^{\otimes l}\right) \right) \left( x\right) \\
&&+\sum_{l=1}^{\left[ p\right] }\sum_{j=0}^{\left[ p\right] -l}\frac{1}{l!}%
\frac{1}{j!}\left( D^{l+j}\varphi \right) \left( h_{s}\right) \left( \left(
h_{t}-h_{s}\right) ^{\otimes j}-\beta \left( X_{s}\right) \left(
X_{s,t}\right) ^{\otimes j}\right) \left( \prod\nolimits_{\leq \left[ p%
\right] }\left( \beta _{1}\left( X_{t}\right) ^{\otimes l}\right) \right)
\left( x\right) \\
&&+\sum_{l=1}^{\left[ p\right] }\frac{1}{l!}\left( D^{l}\varphi \right)
\left( h_{s}\right) \prod\nolimits_{\leq \left[ p\right] }\left( \left(
\beta _{1}\left( X_{s}\right) \left( X_{s,t}\right) +\beta _{1}\left(
X_{t}\right) \right) ^{\otimes l}-\left( \beta _{1}\left( X_{s}\right)
\left( X_{s,t}\right) +\left( \beta _{1}\left( X_{s}\right) \right)
_{X_{s,t}}\right) ^{\otimes l}\right) \left( x\right)
\end{eqnarray*}%
for $x\in G\left( V\right) $. Then the estimate $\left( \ref{Composition
operator norm}\right) $ follows from Theorem \ref{Proposition continuity of
integral w.r.t. one-form}. Based on the comparison of local expansions,%
\begin{equation*}
\int_{r=0}^{t}\beta \left( X_{r}\right) dX_{r}=\varphi \left( h_{t}\right)
-\varphi \left( 0\right) ,t\in \left[ 0,1\right] \text{.}
\end{equation*}
\end{proof}

\subsection{Iterated Integration}

Let $U_{i},i=1,2$ be two Banach spaces. For $\phi _{i}\in
E_{X_{t}}^{U_{i}},i=1,2$, define the `truncated iterated integration':%
\begin{eqnarray*}
\prod\nolimits_{\leq \left[ p\right] }\left( \phi _{1}\succ \phi _{2}\right)
&\in &E_{X_{t}}^{U_{1}\otimes U_{2}} \\
x &\mapsto &\sum_{\substack{ k_{1}+k_{2}\leq \left[ p\right]  \\ %
k_{i}=1,\dots \left[ p\right] }}\left( \phi _{1}^{k_{1}}\otimes \phi
_{2}^{k_{2}}\right) \left( \left( 1_{k_{1}}\succ 1_{k_{2}}\right) \cdot
\left( x^{k_{1}+k_{2}}\right) \right)
\end{eqnarray*}%
for $x\in G\left( V\right) $, where $\phi _{i}=\sum_{k=1}^{\left[ p\right]
}\phi _{i}^{k},\phi _{i}^{k}\in L\left( V^{\otimes k},U_{i}\right) $, $%
x=\sum_{k\geq 0}x^{k},x^{k}\in V^{\otimes k}$.

\begin{proposition}
\label{Proposition stability under iterated integration}For $i=1,2$, let $%
\beta _{i}\in \left( X,E_{X}^{U_{i}}\right) $ be a slowly-varying one-form
such that $\left\Vert \beta _{i}\right\Vert _{\theta _{i}}^{\omega
_{i}}<\infty $ for a control $\omega _{i}$ and $\theta _{i}>1$. Define $%
\beta \in \left( X,E_{X}^{U_{1}\otimes U_{2}}\right) $ as%
\begin{equation*}
\beta \left( X_{t}\right) :=\left( \int_{r=0}^{t}\beta _{1}\left(
X_{r}\right) dX_{r}\right) \otimes \beta _{2}\left( X_{t}\right)
+\prod\nolimits_{\leq \left[ p\right] }\left( \beta _{1}\left( X_{t}\right)
\succ \beta _{2}\left( X_{t}\right) \right)
\end{equation*}%
for $t\in \left[ 0,1\right] $. Then with $\omega :=\omega _{1}+\omega
_{2}+\left\Vert X\right\Vert _{p-var}^{p}$ and $\theta :=\min \left( \theta
_{1},\theta _{2}\right) $,%
\begin{equation}
\left\Vert \beta \right\Vert _{\theta }^{\omega }\leq C_{p,\theta ,\omega
\left( 0,1\right) }\left\Vert \beta _{1}\right\Vert _{\theta _{1}}^{\omega
_{1}}\left\Vert \beta _{2}\right\Vert _{\theta _{2}}^{\omega _{2}}\text{.}
\label{Iterated Integration operator norm}
\end{equation}
\end{proposition}

\begin{remark}
For $t\in \left[ 0,1\right] $, set $\beta _{2}\left( X_{t}\right) \left(
x\right) :=x^{1}$ for $x\in G\left( V\right) ,x=\sum_{k\geq 0}x^{k},x^{k}\in
V^{\otimes k}$. Then the $\beta $ defined above corresponds to the
integration of $\beta _{1}$, and integration is a continuous operation on
slowly-varying one-forms.
\end{remark}

\begin{proof}
Based on Proposition \ref{Proposition expression of iterated integration of
permutations}, for $\phi ^{i}\in E_{X_{t}}^{U_{i}},i=1,2$ and $a\in G\left(
V\right) $,%
\begin{equation*}
\left( \left( \phi ^{1}\right) _{a}\succ \left( \phi ^{2}\right) _{a}\right)
\left( x\right) =\left( \phi ^{1}\succ \phi ^{2}\right) _{a}\left( x\right)
-\phi ^{1}\left( a\right) \otimes \left( \phi ^{2}\right) _{a}\left( x\right)
\end{equation*}%
for $x\in G\left( V\right) $.

Denote $X_{s,t}:=X_{s}^{-1}X_{t}$. For $s\leq t$,%
\begin{eqnarray*}
&&\left( \beta \left( X_{t}\right) -\left( \beta \left( X_{s}\right) \right)
_{X_{s,t}}\right) \left( x\right) \\
&=&\int_{r=0}^{s}\beta _{1}\left( X_{r}\right) dX_{r}\otimes \left( \beta
_{2}\left( X_{t}\right) -\left( \beta _{2}\left( X_{s}\right) \right)
_{X_{s,t}}\right) \left( x\right) \\
&&+\left( \int_{r=s}^{t}\beta _{1}\left( X_{r}\right) dX_{r}-\beta
_{1}\left( X_{s}\right) \left( X_{s,t}\right) \right) dX_{r}\otimes \beta
_{2}\left( X_{t}\right) \left( x\right) \\
&&+\beta _{1}\left( X_{s}\right) \left( X_{s,t}\right) \otimes \left( \beta
_{2}\left( X_{t}\right) -\left( \beta _{2}\left( X_{s}\right) \right)
_{X_{s,t}}\right) \left( x\right) \\
&&+\prod\nolimits_{\leq \left[ p\right] }\left( \left( \beta _{1}\left(
X_{t}\right) \succ \beta _{2}\left( X_{t}\right) \right) -\left( \left(
\beta _{1}\left( X_{s}\right) \right) _{X_{s,t}}\succ \left( \beta
_{2}\left( X_{s}\right) \right) _{X_{s,t}}\right) \right) \left( x\right)
\end{eqnarray*}%
for $x\in G\left( V\right) $. Then the estimate $\left( \ref{Iterated
Integration operator norm}\right) $ holds based on the definition of the
operator norm and Theorem \ref{Proposition continuity of integral w.r.t.
one-form}.
\end{proof}

\section*{Acknowledgement}

The author would like to express sincere gratitude to Prof. Terry Lyons for
numerous inspiring discussions that eventually lead to this paper. The
author also would like to thank Prof. Martin Hairer, Sina Nejad, Dr. Horatio
Boedihardjo, Dr. Xi Geng, Dr. Ilya Chevyrev and Vlad Margarint for
discussions and suggestions on (an earlier version of) the paper.

\bibliographystyle{abbrv}
\bibliography{acompat,roughpath}

\end{document}